\documentclass[pdflatex]{sn-jnl}
\usepackage{float}
\usepackage{verbatim}
\usepackage{graphicx}
\usepackage{multirow}
\usepackage{amsmath,amssymb,amsfonts}
\usepackage{amsthm}
\usepackage{mathrsfs}

\usepackage[title]{appendix}
\usepackage{xcolor}
\usepackage{textcomp}
\usepackage{manyfoot}
\usepackage{booktabs}
\usepackage{algorithm}
\usepackage{algorithmicx}
\usepackage{tikz}
\usepackage{algpseudocode}
\usepackage{listings}
\usepackage{pgfplots}
\pgfplotsset{compat=1.18}
\newtheorem{Theorem}{Theorem}
\newtheorem{Proposition}[Theorem]{Proposition}
\newtheorem{Remark}[Theorem]{Remark}
\newtheorem{Lemma}[Theorem]{Lemma}
\newtheorem{Corollary}[Theorem]{Corollary}
\newtheorem{definition}[Theorem]{Definition}
\newtheorem{Example}[Theorem]{Example}
\newtheorem{remark}[Theorem]{Remark}

\begin{document}
\title[Article Title]{Busemann-coupling subdifferentials and Fenchel biconjugation on Hadamard manifolds}


\author*[1]{\sur{G. C. Bento}}\email{glaydston@ufg.br}

\author[2]{ \sur{J. X.  Cruz Neto}}\email{jxavier@ufpi.edu.br }
\equalcont{These authors contributed equally to this work.}

\author[2]{ \sur{I. D. L.  Melo}}\email{italodowell@ufpi.edu.br }
\equalcont{These authors contributed equally to this work.}

\affil*[1]{\orgdiv{Institute of Mathematics and Statistics, IME}, \orgname{Federal University of Goi\'as}, \orgaddress{\city{Goi\^ania}, \postcode{74690-900}, \state{Goi\'as}, \country{Brazil}}}

\affil[2]{\orgdiv{Departamento de Matemática}, \orgname{Universidade Federal do Piauí}, \orgaddress{\city{Teresina}, \postcode{64048-550}, \state{Piauí}, \country{Brasil}}}



\abstract{ 
We introduce an intrinsic subdifferential on Hadamard manifolds induced
by a base-point-dependent Busemann coupling, motivated by the scarcity
of nonconstant affine functions on non-Euclidean Hadamard manifolds. For
proper functions, it characterizes exactly the primal-dual equality
pairs in the Fenchel-Young inequality of Bento, Cruz Neto, and Melo
(\emph{Appl. Math. Optim.} 88:83, 2023). A
dual-point--tangent-vector bijection yields Fermat's rule and isometry
covariance. We characterize everywhere nonemptiness for finite-valued functions by pointwise-attained \(H_p\)-envelope representations, encompassing the
radial models with explicit subdifferentials and the nonsmooth distance
from the base point. In hyperbolic space, base-point rigidity shows
that, for functions differentiable at the base point, nonemptiness is
equivalent to global minimality and vanishing of the corresponding
Fenchel-Young gap.
Two geodesically convex functions on the Poincaré disk prove strict
noninclusions between our construction and fixed-direction Busemann
subdifferentials, although both recover the classical Euclidean
subdifferential. This separation is explained by coupling orientation.
In constant negative curvature, a strictly increasing scalar profile
describes the asymmetry and defines an intrinsic measure of nonlinearity. We compute this measure at arbitrary radius and derive a lower bound and two-sided estimates for the biconjugation gap; at unit radius, the exact curvature-dependent bound is attained by deviations of both signs.
}




\maketitle

\section{Introduction}

Affine functions are fundamental to many areas of mathematics and
optimization because they provide global models, supporting minorants, and
natural primal--dual couplings. In Euclidean space, given
\(x,z\in\mathbb{R}^n\), the function
\begin{equation}\label{JCA-Euclidean}
    \mathbb{R}^n\ni y\longmapsto
    g_z^x(y):=\langle x-z,y-z\rangle
\end{equation}
is affine in \(y\), and linear in the translated variable \(y-z\). Its
affine and convex structure is used, for example, in the regularization of
the equilibrium resolvent
\begin{equation}\label{Newbi-Function}
    J_{\lambda}^F(z)
    :=
    \bigl\{
        x\in\Omega:
        \lambda F(x,y)-g_z^x(y)\geq0,\quad y\in\Omega
    \bigr\},
    \qquad \lambda>0,
\end{equation}
where the associated equilibrium problem consists of finding
\(x^*\in\Omega\) such that
\[
    F(x^*,y)\geq0,
    \qquad y\in\Omega.
\]
Here, \(\Omega\subset\mathbb{R}^n\) is a nonempty closed convex set and
\(F:\Omega\times\Omega\to\mathbb{R}\) is a bifunction satisfying
\(F(x,x)=0\) for every \(x\in\Omega\); see, for instance,
\cite{combettes2005equilibrium}. For equilibrium problems in other settings,
see, e.g.,
\cite{ansari1999existence,chadli2002equilibrium,iusem2003new,
iusem2009certain,batista2015existence,niculescu2009fan}.

The same coupling is central to Fenchel conjugation. Indeed, in \cite{bento2023fenchel}, we introduced the following more general, base-point-dependent formulation: for \(z\in\mathbb{R}^n\) and \(f:\mathbb{R}^n\to\mathbb{R}\cup\{+\infty\}\), define
\begin{equation}\label{conjugateReal1}
f_z^*(x)
:=
\sup\bigl\{
g_z^x(y)-f(y): y\in\mathbb{R}^n
\bigr\}.
\end{equation}
This formulation incorporates an arbitrary reference point \(z\). When \(z=0\), it reduces to the classical Fenchel conjugate; see \cite[Definition~13.1, p.~181]{bauschke2011convex}.
On a Riemannian manifold, the most immediate analogue of
\eqref{JCA-Euclidean} is
\begin{equation}\label{JCA-2017}
    M\ni y\longmapsto
    g_z^x(y):=
    \left\langle\exp_z^{-1}x,\exp_z^{-1}y\right\rangle,
    \qquad x,z\in M.
\end{equation}
However, the Euclidean affine structure is not preserved by this
construction. It was proved in \cite{kristaly2016what} that the function in
\eqref{JCA-2017} is affine if and only if \(M\) is isometric to Euclidean
space, while \cite{cruzneto2017note} shows that it need not even be convex on
a Hadamard manifold with negative sectional curvature. More fundamentally,
building on \cite[Theorem~2.2, p.~299]{udriste1994convex}, Bento, Cruz Neto,
and Melo proved in \cite{bento2023fenchel} that nonconstant affine functions
are unavailable on Hadamard manifolds with strictly negative sectional
curvature. Thus, the affine supports underlying classical equilibrium
regularization and Fenchel duality have no direct intrinsic counterpart in
this geometric setting.

This obstruction provided one of the original motivations for the research
program initiated by Bento, Cruz Neto, and Melo: to identify intrinsic
geometric substitutes for affine functions on Hadamard manifolds. Busemann
functions arise naturally for this purpose. Let
\(\gamma:[0,+\infty)\to M\) be a unit-speed geodesic ray emanating from
\(p\in M\), and let \(v=\dot\gamma(0)\in T_pM\), with \(\|v\|=1\). The
associated Busemann function \(B_v^p:M\to\mathbb{R}\) is defined by
\begin{equation*}\label{defib}
    B_v^p(x)
    :=
    \lim_{t\to+\infty}
    \bigl(d(x,\gamma(t))-t\bigr).
\end{equation*}
Busemann functions are convex and \(1\)-Lipschitz on Hadamard manifolds and,
in the Euclidean setting, reduce to affine functions. They therefore retain
precisely the global geometric features required to replace affine supports
in a fully intrinsic manner. For general discussions and examples, see
\cite{bridson2013metric,ballmann2013manifolds}.

Originating in the work of Herbert Busemann
\cite{busemann1955geometry}, Busemann functions have been extensively
studied in geometry and analysis; see, for instance,
\cite{busemann1993novel,li1987positive,sormani1998busemann,
shiohama1979busemann} and the references therein. In optimization on
Hadamard manifolds, these functions have been employed in connection
with two related problems: the construction of convex regularizations
for equilibrium resolvents and the development of Fenchel-type
conjugation in this geometric setting. More specifically,
\cite{bento2022combinatorial,bento2023fenchel} proposed,
respectively, the constructions
\begin{equation}\label{Bresolvent}
J_{\lambda}^F(x)
:=
\left\{
z\in\Omega:
\lambda F(z,y)+d(z,x)B_v^z(y)\geq0,\quad y\in\Omega
\right\},
\quad
\lambda>0,\quad
v=\frac{\exp_z^{-1}x}{d(z,x)},
\end{equation}
and
\begin{equation}\label{conjugateRiemannian1}
f_p^*(x)
=
\sup\left\{
tB_v^p(x)-f(\exp_p(-tv)):
t\geq0,\ v\in T_pM,\ |v|=1
\right\},
\end{equation}
\medskip

\noindent as counterparts on Hadamard manifolds of
\eqref{Newbi-Function} and \eqref{conjugateReal1}. In
\eqref{Bresolvent}, \(\Omega\subset M\) is a nonempty closed convex set
and \(F:\Omega\times\Omega\to\mathbb{R}\) satisfies \(F(x,x)=0\) for
every \(x\in\Omega\); in \eqref{conjugateRiemannian1}, \(p\in M\) and
\(f\to\mathbb{R}\cup{+\infty}\). Independently and contemporaneously
with \cite{bento2023fenchel}, Hirai \cite{hirai2023convex} developed an
asymptotic Legendre-Fenchel conjugation on Hadamard spaces using
Busemann functions.

These contributions form part of a broader line of research exploring
Busemann functions as geometric substitutes for affine functions in
continuous optimization on nonpositively curved spaces. Subsequent and
related developments include proximal and extragradient methods for
equilibrium problems, Busemann-based subgradient methods,
horospherical convexity, and barrier geometry; see, for instance,
\cite{bento2024new,sharmaa2025regularized,sharma2026bregman,
goodwin2026stochastic,criscitiello2025horospherically,
criscitiello2026negative,ferreira2026subdifferential}
and the references therein.

Against this background, the main contribution of the present paper is to
develop a new subdifferential framework and use it to address
problems identified in our initial investigations. We recall and contextualize
these problems below.
\medskip

\noindent\textbf{Problem 1.}
In \cite[Theorem~5]{bento2023fenchel}, we established the Fenchel-Young
inequality
\begin{equation}\label{Fechel-Young1001}
    f(x)+f_p^*(y)
    \geq
    d(p,x)B_{v_x}^p(y),
    \qquad
    v_x:=
    -\frac{\exp_p^{-1}x}{d(x,p)},
\end{equation}
for \(x\neq p\), with the coupling term on the right understood to be zero
when \(x=p\). The corresponding nonnegative quantity
\[
    \mathcal G_p(x,y)
    :=
    f(x)+f_p^*(y)-d(p,x)B_{v_x}^p(y)
    \geq0
\]
is the associated primal--dual gap and measures the compatibility between
the primal point \(x\) and the dual point \(y\). We conjectured that a
subdifferential framework based on Busemann functions could characterize the
conditions under which this gap vanishes:
\[
    f(x)+f_p^*(y)
    =
    d(p,x)B_{v_x}^p(y).
\]
When equality holds, \(x\) attains the supremum defining \(f_p^*(y)\), and
\(y\) provides an exact global support for \(f\) at \(x\).
This interpretation is particularly transparent when \(M=\mathbb{R}^n\),
because
\[
    d(p,x)B_{v_x}^p(y)
    =
    \langle x-p,y-p\rangle.
\]
In this case, equality in the Fenchel-Young inequality is equivalent to
\(y-p\in\partial f(x)\), that is,
\begin{equation}\label{acoplamento1}
    f(z)
    \geq
    f(x)+\langle y-p,z-x\rangle
    =
    f(x)+\langle y-p,z-p\rangle
    -\langle y-p,x-p\rangle,
    \qquad z\in\mathbb{R}^n.
\end{equation}
Thus, the first problem is to develop an intrinsic Busemann-based
subdifferential that characterizes equality in the Fenchel-Young inequality
and establishes an exact correspondence between the associated primal and
dual objects. 
Beyond this primal-dual characterization, we give an exact functional
description of everywhere nonemptiness. For finite-valued functions,
\[
\partial B^p f(x)\neq\varnothing
\quad\text{for every }x\in M,
\]
if and only if \(f\) is a pointwise-attained \(H_p\)-envelope. This class includes pointwise maxima of finitely many \(H_p\)-sections, the radial convex
functions treated here, and distance from the base point. Complementarily,
a rigidity theorem in hyperbolic space shows that, for functions
differentiable at \(p\),
\[
\partial B^p f(p)\neq\varnothing
\quad\Longleftrightarrow\quad
p\in\operatorname*{argmin}f,
\]
equivalently, the Fenchel-Young gap vanishes at the dual point \(y=p\).
Thus, nonemptiness is governed both by the global \(H_p\)-support
representation of the function and, at the distinguished base point,
by its variational compatibility with the coupling.

During the development of this research program, some notions of
subdifferential based on Busemann functions were proposed in the related
settings of Hadamard spaces and Hadamard manifolds; see
\cite{criscitiello2025horospherically,goodwin2024subgradient,
ferreira2026subdifferential}. After accounting for differences in sign
conventions, normalizations, and parametrizations, the Busemann
subdifferential of \cite{goodwin2024subgradient}, the horospherical
subdifferential of \cite{criscitiello2025horospherically}, and the
\(B\)-subdifferential discussed in
\cite[Remarks~2 and~3]{ferreira2026subdifferential} are equivalent at the
level of their underlying global support inequality: all three are generated
by a single scaled Busemann function associated with a fixed ideal direction.
Although their surrounding theories and applications differ, they therefore
impose the same type of support condition on a Hadamard manifold.

The \(B^p\)-subdifferential introduced here has a different structure. Its
support condition is generated by the base-point-dependent two-point
coupling
\begin{equation}\label{eq:111}
\begin{aligned}
    d(p,z)B^p_{-\frac{\exp_p^{-1}z}{d(p,z)}}(y)
    -
    d(p,x)B^p_{-\frac{\exp_p^{-1}x}{d(p,x)}}(y),
\end{aligned}
\end{equation}
whenever \(x,z\neq p\). Unlike the fixed-direction constructions,
\eqref{eq:111} is not generated by a single Busemann function of \(z\):
both the weight and the direction in its first term depend on \(z\), whereas
the Busemann function is evaluated at the dual point \(y\). In the Euclidean
setting, \eqref{eq:111} reduces to the classical affine support in
\eqref{acoplamento1}:
\[
    \langle y-p,z-p\rangle-\langle y-p,x-p\rangle
    =
    \langle y-p,z-x\rangle.
\]
Outside the flat setting, the subdifferential proposed here is genuinely distinct from the fixed-direction constructions discussed above. This
distinction also reflects their different purposes. The fixed-direction
notions provide global support mechanisms used in first-order,
horospherical, and projection--proximal methods. By contrast, our
\(B^p\)-subdifferential is defined for arbitrary proper functions and is
intrinsically linked to the Fenchel conjugate introduced in
\cite{bento2023fenchel}: it characterizes equality in the corresponding
Fenchel-Young inequality, yields a precise primal--dual correspondence, and
admits natural transformation rules under isometries. Thus, the distinction
between these frameworks arises from the geometry of their underlying
couplings and cannot be removed by changes of sign convention,
normalization, or parametrization. 

This structural comparison reveals a common geometric source that connects
the first problem with the biconjugation problem considered next. Indeed, if
\[
    y=\exp_p(-s\eta),
    \qquad s>0,\quad\|\eta\|=1,
\]
then a fixed-direction support can be written as
\(
    sB_\eta^p(z)=H_p(y,z),
\)
whereas the support defining our subdifferential involves
\[
    H_p(z,y)-H_p(x,y).
\]
Consequently, their structural difference is governed by the lack of
reciprocity
\[
    H_p(z,y)-H_p(y,z).
\]
The maximal magnitude of this asymmetry had already been introduced in \cite[Section~6]{bento2023fenchel} as an  intrinsic measure of nonlinearity, albeit only in an embryonic form and without a systematic analysis, namely,
\[
    -\mathcal N^p(y)
    :=
    \sup_{z\in M}
    \bigl\{H_p(z,y)-H_p(y,z)\bigr\}\geq0.
\]
Thus, \(-\mathcal N^p\) provides a common \emph{measure of nonlinearity}: it
both distinguishes our subdifferential from the fixed-direction Busemann
constructions and controls the deviation from exact biconjugation. In the
Euclidean case, \(H_p\) is symmetric, so this measure vanishes and both
distinctions disappear.

\medskip
\noindent\textbf{Problem 2.}
Let \(M\) be a Hadamard manifold with constant sectional curvature
\(-k^2\), where \(k>0\), let \(p\in M\), and let
\(\bar\eta\in T_pM\) be a unit vector. In
\cite[Theorem~3]{bento2023fenchel}, we established that
\begin{equation}\label{FM2023-16}
\bigl(B_{\bar\eta}^p\bigr)^{**}
=
B_{\bar\eta}^p
-
\frac{1}{k}
\log\left(\cosh\left(\frac{k}{2}\right)\right).
\end{equation}
Motivated by this identity, we conjectured in
\cite[Section~6]{bento2023fenchel} a Fenchel--Moreau-type estimate of the
form
\begin{equation}\label{FM1}
f_p^{**}\leq f\leq f_p^{**}+C(f,M),
\end{equation}
with a curvature-dependent error term.
This biconjugation problem should be distinguished from the
Fenchel--Moreau-type result of Goodwin et al.
\cite[Proposition~3.5]{goodwin2024subgradient}. In their framework, the
biconjugate is obtained by composing a conjugation from functions on the
Hadamard space to functions on its boundary cone with the corresponding
adjoint transform in the reverse direction. Thus, it is not obtained by
iterating a single conjugation operator, and the adjoint pair preserves
the primal--dual ordering of its coupling. By contrast, \(f_p^{**}\) is
defined by reapplying the same conjugation on functions over \(M\), which
interchanges the arguments of the generally nonsymmetric coupling
\(H_p\). Consequently, the two constructions address different
biconjugation problems; see
Remark~\ref{rem:Goodwin-adjoint-biconjugation}.


The logarithmic-hyperbolic expression in \eqref{FM2023-16} already
contained implicitly the scalar profile
\begin{equation}\label{eq:functionAux}
t\longmapsto\varphi(t):=\frac{\log(\cosh t)}{t}
\end{equation}
One of the contributions of the present paper is to identify this
function explicitly and reveal its
role as a quantitative profile of the nonreciprocity of \(H_p\).
Indeed, the correction term in \eqref{FM2023-16} can be written as

$$
    \frac{1}{k}\log\!\left(\cosh\!\left(\frac{k}{2}\right)\right)
    =
    \frac{1}{2}\varphi\!\left(\frac{k}{2}\right).
$$
More fundamentally, for orthogonal unit vectors
\(\eta,\theta\in T_pM\), \(r,s>0\), and

$$
    y=\exp_p(-s\eta),
    \qquad
    z=\exp_p(r\theta),
$$
we prove the exact identity

$$
    H_p(z,y)-H_p(y,z)
    =
    rs\bigl[\varphi(ks)-\varphi(kr)\bigr].
$$
Since \(\varphi\) is strictly increasing, this representation determines
the sign of the asymmetry solely from the relative radial positions of
\(y\) and \(z\), while its magnitude records the combined effects of
curvature and distance. Thus, \(\varphi\) is not merely a
reparametrization of the constant found in
\cite[Theorem~3]{bento2023fenchel}; it exposes the geometric mechanism
behind that constant and makes the obstruction to the conjectured
one-sided ordering in \eqref{FM1} transparent.
Indeed, the analysis in Section~\ref{Sec4} shows that the ordering in
\eqref{FM1} does not hold in general: because \(H_p\) is not symmetric,
the difference \(f-f_p^{**}\) may have either sign. The second problem
must therefore be reformulated as the search for an intrinsic two-sided
control of the biconjugation gap. We prove that, for every proper
function \(f\) and every \(x\in\operatorname{dom}f\),

$$
    \mathcal N^p(x)
    \leq
    f(x)-f_p^{**}(x),
$$
and, whenever \(v\in\partial B^p f(x)\) and \(y=Y_p(x,v)\),

$$
    \mathcal N^p(x)
    \leq
    f(x)-f_p^{**}(x)
    \leq
    H_p(x,y)-H_p(y,x)
    \leq
    -\mathcal N^p(y).
$$
Moreover, if $M$ has constant sectional curvature 
-$k^2$, $k > 0$, and \(d(y,p)=1\), one has
$$
    -\mathcal N^p(y)
    =
    C_k
    :=
    \frac{1}{k}
    \log\!\left(\cosh\!\left(\frac{k}{2}\right)\right)
    =
    \frac{1}{2}\varphi\!\left(\frac{k}{2}\right),
$$
and this constant is attained by biconjugation deviations of both signs.
Hence, \(-\mathcal N^p\) is precisely the geometric quantity that
controls the failure of exact biconjugation in constant negative
curvature, whereas \(\varphi\) provides the scalar profile governing the
underlying asymmetry.

The remainder of the paper is organized as follows. In
Sect.~\ref{sec:geometric-preliminaries}, we recall the geometric ingredients needed to analyze the base-point-dependent Busemann coupling. In Sect.~\ref{Sec3}, we
introduce the \(B^p\)-subdifferential and establish its main theoretical
connections with the Fenchel conjugate. In Sect.~\ref{Sec4}, we analyze
Fenchel biconjugation and the associated measure of nonlinearity. Finally,
Sect.~\ref{Sec6} concludes the paper with final remarks and directions for
future research.

\section{Busemann geometry and coupling asymmetry}
\label{sec:geometric-preliminaries}
This section recalls the geometric ingredients needed to analyze the
base-point-dependent Busemann coupling, denoted by \(H_p\). We first
recall basic homothetic scaling relations, the hyperbolic law of cosines,
and the curvature-dependent expression for Busemann functions established
in \cite{bento2023fenchel}. We then introduce an auxiliary scalar function
that yields an explicit radial description of the asymmetry of \(H_p\)
in constant negative sectional curvature. This formula will provide the
geometric basis for both the comparison of Busemann subdifferential
frameworks and the subsequent analysis of Fenchel biconjugation. Unlike
its Euclidean counterpart, \(H_p\) is generally not symmetric, and this
lack of reciprocity is a genuinely nonlinear feature of the underlying
geometry.

A complete and simply connected Riemannian manifold with nonpositive
sectional curvature is called a Hadamard manifold. 
We shall also use the behavior of the relevant geometric objects under
constant conformal rescalings. Let \(g\) denote the Riemannian metric of
\(M\), and, for \(\lambda>0\), consider the homothetic metric
\begin{equation}
\label{eq:homothetic-metric}
    g_{\lambda}:=\frac{1}{\lambda^2}g.
\end{equation}
We denote the resulting Riemannian manifold by \(M_{\lambda}\), and use
the subscript \(\lambda\) for its associated geometric objects. Since
the conformal factor is constant, the Levi--Civita connections and
exponential mappings associated with \(g\) and \(g_{\lambda}\) coincide.
Moreover,
\begin{equation}
\label{eq:homothetic-geometric-scaling}
    \|u\|_{\lambda}
    =
    \frac{1}{\lambda}\|u\|,
    \qquad
    d_{\lambda}(x,z)
    =
    \frac{1}{\lambda}d(x,z),
    \qquad
    K_{g_{\lambda}}(\Pi)
    =
    \lambda^2K_g(\Pi),
\end{equation}
for every \(u\in TM\), \(x,z\in M\), and every two-dimensional subspace
\(\Pi\subset T_xM\). Angles are unchanged by this rescaling. These facts
are recalled in \cite{bento2023fenchel}.
In particular, if \(M\) has constant sectional curvature \(-1\), then
\(M_k\) has constant sectional curvature \(-k^2\), and lengths in the
rescaled metric are divided by \(k\). 
An important model is the hyperbolic space of constant sectional curvature
\(-k^2\), where \(k>0\). The previous observation explains the
appearance of the curvature parameter \(k\) in the hyperbolic law of
cosines below. If \(a,b,c\) are the side lengths of a geodesic
triangle and \(\alpha\) is the angle opposite the side of length \(a\), then
the hyperbolic law of cosines takes the form
\begin{equation}
\label{eq:hyperbolic-law-cosines-prelim}
    \cosh(ka)
    =
    \cosh(kb)\cosh(kc)
    -
    \sinh(kb)\sinh(kc)\cos\alpha.
\end{equation}
In particular, if \(\alpha=\pi/2\), then
\begin{equation}
\label{eq:hyperbolic-law-cosines-orthogonal}
    \cosh(ka)
    =
    \cosh(kb)\cosh(kc).
\end{equation}
Although only the equality corresponding to constant sectional curvature is required in this paper, it is worth noting that the comparison versions of \eqref{eq:hyperbolic-law-cosines-prelim} for sectional curvature bounded from below or from above are also provided in \cite[Proposition~7]{bento2023fenchel}.

The following exact expression, established in
\cite[Lemma~1]{bento2023fenchel}, follows from the hyperbolic law of cosines and will be useful throughout this paper.
\begin{Lemma}[Busemann functions in constant negative curvature]
\label{lem:Busemann-constant-curvature}
Let \(M\) be a Hadamard manifold with constant sectional curvature
\(-k^2\), where \(k>0\). Let \(u,v\in T_pM\) be unit vectors and set
\(
    x=\exp_p(-tv),\)
    \(t=d(p,x)>0.
\)
If \(\alpha\) denotes the angle between \(-v\) and \(u\), then
\begin{equation*}
\label{eq:Busemann-constant-curvature}
    B_u^p(x)
    =
    \frac{1}{k}
    \log\bigl(
        \cosh(kt)-\sinh(kt)\cos\alpha
    \bigr).
\end{equation*}
\end{Lemma}
For any $q\in M$ and any
nonzero vector $u\in T_qM$, we denote by $B_u^q$ the Busemann function
associated with the unit-speed geodesic ray issued from $q$ in the direction
$u/\|u\|$, i.e., $B_u^q=B_{u/\|u\|}^q$. 
Busemann functions also satisfy a simple scaling rule under the
homothetic change \eqref{eq:homothetic-metric}. More precisely, let
\(v\in T_pM\) be a unit vector with respect to \(g\), and set
\(\widetilde v:=\lambda v\), which is a unit vector with respect to
\(g_{\lambda}\). If \(B_{\widetilde v}^{p,\lambda}\) denotes the
corresponding Busemann function on \(M_{\lambda}\), then
\cite[Proposition~8]{bento2023fenchel} gives
\begin{equation}
\label{eq:Busemann-homothetic-scaling}
    B_{\widetilde v}^{p,\lambda}(x)
    =
    \frac{1}{\lambda}B_v^p(x),
    \qquad x\in M.
\end{equation}
Given a fixed base point
$p\in M$, define the Busemann coupling $H_p:M\times M\to\mathbb{R}$ by
\begin{equation}\label{eq:Busemann-coupling}
H_p(z,y):=
\begin{cases}
\displaystyle
d(p,z)B^p_{-\exp_p^{-1}z}(y), & z\neq p,\\[2mm]
0, & z=p.
\end{cases}
\end{equation}
Equivalently, for $z\neq p$,
\[
    H_p(z,y)
    =d(p,z)B^p_{-\frac{\exp_p^{-1}z}{d(p,z)}}(y).
\]
In the Euclidean setting, one has
\(
    H_p(z,y)=\langle z-p,y-p\rangle.
\)
Let \(H_p^{\lambda}\) denote the Busemann coupling associated with the
rescaled manifold \(M_{\lambda}\). Combining
\eqref{eq:homothetic-geometric-scaling} and
\eqref{eq:Busemann-homothetic-scaling}, we obtain
\begin{equation}
\label{eq:Busemann-coupling-homothetic-scaling}
    H_p^{\lambda}(z,y)
    =
    \frac{1}{\lambda^2}H_p(z,y),
    \qquad z,y\in M.
\end{equation}
This scaling identity will later allow us to transfer the unit-radius
formula for the intrinsic nonlinearity measure to points at arbitrary
distance from the base point.

Unlike the Euclidean bilinear pairing, \(H_p\) is generally not symmetric
on a nonflat Hadamard manifold. To describe its radial asymmetry in
constant negative curvature, consider the function
\(
\varphi(t):=\log(\cosh t)/t
\) introduced in \eqref{eq:functionAux} 
which is strictly increasing. 

We can now express the asymmetry of \(H_p\) along orthogonal radial
directions directly in terms of \(\varphi(\cdot)\).

\begin{Proposition}[Radial asymmetry of the Busemann coupling]
\label{prop:Hp-asymmetry-phi}
Let \(M\) be a Hadamard manifold with constant sectional curvature
\(-k^2\), where \(k>0\), and let \(p\in M\). Given orthogonal unit vectors
\(\eta,\theta\in T_pM\) and \(r,s>0\), define
\(
    y:=\exp_p(-s\eta),
    \)
    \(z:=\exp_p(r\theta).
\)
Then,
\begin{equation}
\label{eq:Hp-orthogonal-values}
    H_p(y,z)
    =
    \frac{s}{k}\log\bigl(\cosh(kr)\bigr),\qquad H_p(z,y)
    =
    \frac{r}{k}\log\bigl(\cosh(ks)\bigr)
\end{equation}
Consequently,
\begin{equation}
\label{eq:Hp-asymmetry-phi}
    H_p(z,y)-H_p(y,z)
    =
    rs\bigl[\varphi(ks)-\varphi(kr)\bigr].
\end{equation}
In particular, for \( 0<r<s\) one has \(H_p(z,y)>H_p(y,z).
\)
\end{Proposition}
\begin{proof}
The identities in \eqref{eq:Hp-orthogonal-values} follow directly from
Lemma~\ref{lem:Busemann-constant-curvature}, applied to the definitions of
\(y\) and \(z\), together with the orthogonality of the unit vectors
\(\eta\) and \(\theta\). Subtracting the identities in
\eqref{eq:Hp-orthogonal-values} and using the definition of
\(\varphi\) in \eqref{eq:functionAux} yields
\eqref{eq:Hp-asymmetry-phi}. The final assertion follows immediately from
the strict monotonicity of \(\varphi(\cdot)\).
\end{proof}
\begin{Remark}
Identity \eqref{eq:Hp-asymmetry-phi} shows that the lack of reciprocity of
\(H_p\) is already visible in the two-dimensional geometry generated by
the orthogonal directions \(\eta\) and \(\theta\). Moreover, \(\varphi(\cdot)\)
measures the radial variation of the same asymmetry whose global magnitude
will be quantified by the intrinsic nonlinearity measure.
\end{Remark}
In the remainder of this paper, unless otherwise stated, M will denote an n-dimensional Hadamard manifold.

\section{Subdifferential via Busemann functions on Hadamard manifolds}
\label{Sec3}

In this section, we introduce a subdifferential framework based on the
Busemann coupling underlying the Fenchel conjugate proposed in
\cite{bento2023fenchel}. We first establish the exact correspondence between our
$B^{p}$-subdifferential and equality in the Fenchel-Young inequality,
investigate its behavior under isometries, and determine it explicitly for a
class of radial functions.  We then, compare our construction with the
fixed-direction Busemann subdifferentials previously considered in the
literature.

Given $x\in M$ and $v\in T_xM\setminus\{0\}$, let
$\eta_{p,x}(v)\in T_pM$ be the unique unit vector such that the unit-speed
geodesic rays
\[
    \gamma_{(p,\eta_{p,x}(v))}
    \quad\text{and}\quad
    \gamma_{\left(x,-v/\|v\|\right)}
\]
are asymptotic. Define
\begin{equation}\label{eq:Ypxv}
Y_p(x,v):=
\begin{cases}
\displaystyle
\exp_p\bigl(-\|v\|\eta_{p,x}(v)\bigr), & v\neq0,\\[1mm]
p, & v=0.
\end{cases}
\end{equation}

Inspired by the representation subgradient in \eqref{acoplamento1}, we propose the following definition of subdifferential via Busemann functions.
\begin{definition}\label{def:Bp-subdifferential}
Let $f:M\to\mathbb{R}\cup\{+\infty\}$ be proper and let $x\in\emph{dom} f$.
A vector $v\in T_xM$ is called a $B^p$-subgradient of $f$ at $x$ if
\begin{equation}\label{eq:Bp-support}
    f(z)\geq f(x)
    +H_p\bigl(z,Y_p(x,v)\bigr)
    -H_p\bigl(x,Y_p(x,v)\bigr),
    \qquad z\in M.
\end{equation}
The set of all $B^p$-subgradients of $f$ at $x$ is called the
$B^p$-subdifferential of $f$ at $x$ and is denoted by
\(
    \partial B^p f(x).
\)
\end{definition}
In the Euclidean setting, $\eta_{p,x}(v)=-v/\|v\|$ and hence
\[
    Y_p(x,v)=p+v.
\]
Combining this identity with \eqref{acoplamento1}, inequality
\eqref{eq:Bp-support} becomes
\[
\begin{aligned}
    f(z)
    &\geq f(x)
      +\langle y-p,z-p\rangle
      -\langle y-p,x-p\rangle\\
    &=f(x)+\langle y-p,z-x\rangle
     =f(x)+\langle v,z-x\rangle,
\end{aligned}
\]
where $y=p+v$. Thus, Definition~\ref{def:Bp-subdifferential} reduces to the
classical subdifferential in Euclidean spaces.

\subsection{Fenchel-Young equality and the $B^p$-subdifferential}

Next we introduce the inverse correspondence between a dual point and its associated tangent vector. For $y\neq p$, define
\[
    \eta_p(y):=-\frac{\exp_p^{-1}y}{d(p,y)},
\]
and let $u_{x,y}\in T_xM$ be the unique unit vector such that
\(  \gamma_{(p,\eta_p(y))}\)  and   \( \gamma_{(x,-u_{x,y})}
\)
are asymptotic. For each fixed $x\in M$, set
\begin{equation}\label{eq:Vpxy}
y\longmapsto V_p(x,y):=
\begin{cases}
d(p,y)u_{x,y}, & y\neq p,\\
0, & y=p.
\end{cases}
\end{equation}
To formulate Fenchel--Young equality as an intrinsic condition, we first relate dual points in \(M\) to tangent vectors at a fixed point \(x\). The following lemma shows that the mappings \(Y_p(x,\cdot)\) and \(V_p(x,\cdot)\) are mutually inverse and that each coupling \(H_p(\cdot,y)\) admits \(V_p(x,y)\) as its exact \(B^p\)-subgradient at \(x\). This identification provides the geometric bridge used in the central theorem below to characterize attainment of the conjugate supremum by a uniquely determined tangent-space subdifferential condition.
\begin{Lemma}\label{lem:Y-V-inverse-H-support}
Let $p\in M$, and let the mappings $Y_p$ and $V_p$ be defined by
\eqref{eq:Ypxv} and \eqref{eq:Vpxy}, respectively. For every fixed
$x\in M$, the mappings
\[
Y_p(x,\cdot):T_xM\longrightarrow M
    \qquad\text{and}\qquad
    V_p(x,\cdot):M\longrightarrow T_xM
\]
are inverse to each other. More precisely,
\begin{equation*}\label{Inverse1}   V_p\bigl(x,Y_p(x,v)\bigr)=v,
    \qquad v\in T_xM,
\end{equation*}
and
\begin{equation*}\label{Inverse2}    Y_p\bigl(x,V_p(x,y)\bigr)=y,
    \qquad y\in M.
\end{equation*}
Furthermore, for each fixed $y\in M$, the function
\[
    H_y:M\longrightarrow\mathbb{R},
    \qquad
    H_y(z):=H_p(z,y),
\]
satisfies
\begin{equation}\label{eq:Sub101}
    V_p(x,y)\in\partial B^p H_y(x),
    \qquad x\in M.
\end{equation}
In fact, the support inequality defining the $B^p$-subdifferential
holds as an equality at every point of $M$.
\end{Lemma}

\begin{proof}
We first prove that
\[
    V_p\bigl(x,Y_p(x,v)\bigr)=v.
\]
If $v=0$, then the definitions immediately give
\[
    Y_p(x,0)=p
    \qquad\text{and}\qquad
    V_p(x,p)=0,
\]
so the identity follows. Suppose now that $v\neq0$, and set
\[
    \eta:=\eta_{p,x}(v),
    \qquad
    y:=Y_p(x,v)
      =\exp_p\bigl(-\|v\|\eta\bigr).
\]
Since the exponential map at $p$ is a global diffeomorphism, we have
\[
    \exp_p^{-1}y=-\|v\|\eta
    \qquad\text{and}\qquad
    d(p,y)=\|v\|.
\]
Consequently,
\[
    \eta_p(y)
    =-\frac{\exp_p^{-1}y}{d(p,y)}
    =\eta.
\]
By the definition of $\eta_{p,x}(v)$, the rays
\[
    \gamma_{(p,\eta)}
    \qquad\text{and}\qquad
    \gamma_{\left(x,-v/\|v\|\right)}
\]
are asymptotic. On the other hand, $u_{x,y}$ is the unique unit vector
in $T_xM$ for which
\[
    \gamma_{(p,\eta_p(y))}
    \qquad\text{and}\qquad
    \gamma_{(x,-u_{x,y})}
\]
are asymptotic. Since $\eta_p(y)=\eta$, due to the uniqueness of the asymptotic direction from \(x\), it follows that
\(
    u_{x,y}=v/\|v\|, 
\) and, hence,
\[
\begin{aligned}
    V_p\bigl(x,Y_p(x,v)\bigr)
    =V_p(x,y)
    =d(p,y)u_{x,y}
    =\|v\|\frac{v}{\|v\|}
    =v.
\end{aligned}
\]
For proving the converse identity
\(
Y_p\bigl(x,V_p(x,y)\bigr)=y,
\)
first note that the case $y=p$ follows directly from
\(
    V_p(x,p)=0\)
and
    \(Y_p(x,0)=p.
\)
Assume, therefore, that $y\neq p$, and set
\[
    \rho:=d(p,y),\qquad
    \eta:=\eta_p(y)
       =-\frac{\exp_p^{-1}y}{\rho},
    \qquad
    u:=u_{x,y}.
\]
By definition,
\(
    V_p(x,y)=\rho u.
\)
Moreover, the ray $\gamma_{(p,\eta)}$ is asymptotic to
$\gamma_{(x,-u)}$. Since
\[
    -\frac{V_p(x,y)}{\|V_p(x,y)\|}
    =-\frac{\rho u}{\rho}
    =-u,
\]
the uniqueness in the definition of
$\eta_{p,x}\bigl(V_p(x,y)\bigr)$ implies that
\[
    \eta_{p,x}\bigl(V_p(x,y)\bigr)=\eta.
\]
Therefore,
\[
\begin{aligned}
    Y_p\bigl(x,V_p(x,y)\bigr)
    =
    \exp_p\left(
        -\|V_p(x,y)\|
        \eta_{p,x}\bigl(V_p(x,y)\bigr)
    \right)
    =\exp_p(-\rho\eta)
    =\exp_p\bigl(\exp_p^{-1}y\bigr)
    =y,
\end{aligned}
\]
which concludes the proof that $Y_p(x,\cdot)$ and $V_p(x,\cdot)$ are mutually
inverse.

Finally, fix $y\in M$ and let $H_y(z):=H_p(z,y)$. From the inverse
identity just proved, we have
\[
    Y_p\bigl(x,V_p(x,y)\bigr)=y.
\]
Hence, for every $z\in M$,
\[
\begin{aligned}
&H_y(x)
 +H_p\bigl(z,Y_p(x,V_p(x,y))\bigr)
 -H_p\bigl(x,Y_p(x,V_p(x,y))\bigr)\\
&\qquad
=H_p(x,y)+H_p(z,y)-H_p(x,y)\\
&\qquad
=H_p(z,y)
=H_y(z).
\end{aligned}
\]
Thus, the support inequality in
Definition~\ref{def:Bp-subdifferential} holds as an equality for every
$z\in M$, and consequently
\[
    V_p(x,y)\in\partial B^p H_y(x).
\]
When $y=p$, this conclusion remains valid because
$V_p(x,p)=0$ and $H_p(z,p)=0$ for every $z\in M$.
\end{proof}

\begin{Remark}\label{Rem-1001}\;
\begin{itemize}
\item [a.] Note that the argument used in the proof of \eqref{eq:Sub101} does not require \(H_p\) to be symmetric; it depends only on the orientation \(H_p(\cdot,y)\) used in the definition of the \(B^p\)-subdifferential;

\item [b.]Note that in the Euclidean setting, $H_y(z)=\langle z-p,y-p\rangle$ and
\(
    \partial B^p H_y(x)=\{y-p\};
\)
\item [c.] In terms of the Busemann coupling $H_p$, the Fenchel conjugate
and the Fenchel-Young inequality introduced in
\cite{bento2023fenchel} and recalled in
\eqref{conjugateRiemannian1} and \eqref{Fechel-Young1001},
respectively, admit the equivalent formulations
\begin{equation}\label{eq:Fenchel-H}
    f_p^*(y)
    =
    \sup_{z\in M}\bigl\{H_p(z,y)-f(z)\bigr\},
    \qquad y\in M,
\end{equation}
and
\begin{equation*}\label{eq:Fenchel-Young-H}
    f(x)+f_p^*(y)\geq H_p(x,y),
    \qquad x,y\in M.
\end{equation*}
\end{itemize}
\end{Remark}
\medskip

The following theorem provides an intrinsic tangent-space realization of equality in the Fenchel-Young inequality. Beyond the formal
equivalence between attainment of the conjugate supremum and the global support inequality, the theorem gives each zero-gap dual point a unique
geometrically meaningful tangent representation. This correspondence
yields the Fermat rule, underlies the covariance formula under
isometries and the radial calculus.
\begin{Theorem}[Fenchel-Young equality and Fermat rule]
\label{thm:FY-Busemann-characterization}
Let  
\(f:M\to\mathbb{R}\cup\{+\infty\}\) be proper, \(p\in M\) and
\(x\in\operatorname{dom}f\). Then, for every \(y\in M\),
\begin{equation}
\label{eq:FY-Busemann-characterization}
f(x)+f_p^*(y)=H_p(x,y)
\quad\Longleftrightarrow\quad
V_p(x,y)\in\partial B^p f(x).
\end{equation}
In particular, the choice \(y=p\) yields the Fermat rule
\begin{equation}
\label{eq:Fermat-rule-Bp}
x\in\operatorname*{argmin}_{z\in M}f(z)
\quad\Longleftrightarrow\quad
0_x\in\partial B^p f(x)
\quad\Longleftrightarrow\quad
f(x)+f_p^*(p)=H_p(x,p)=0,
\end{equation}
where \(0_x\) denotes the zero vector in \(T_xM\).
\end{Theorem}

\begin{proof}
By \eqref{eq:Fenchel-H}, equality in the Fenchel--Young inequality is
equivalent to

$$
\begin{aligned}
    f(x)+f_p^*(y)=H_p(x,y)
    \quad\Longleftrightarrow\quad
    H_p(x,y)-f(x)
    =
    \sup_{z\in M}\bigl\{H_p(z,y)-f(z)\bigr\}.
\end{aligned}
$$
Since the term corresponding to \(z=x\) holds in the supremum, the
last equality holds if and only if
$$
    H_p(x,y)-f(x)
    \geq
    H_p(z,y)-f(z),
    \qquad z\in M.
$$
Equivalently,
$$
    f(z)
    \geq
    f(x)+H_p(z,y)-H_p(x,y),
    \qquad z\in M.
$$
By Lemma~\ref{lem:Y-V-inverse-H-support},
$$
    Y_p\bigl(x,V_p(x,y)\bigr)=y.
$$
Consequently, the preceding inequality is precisely the support
inequality in Definition~\ref{def:Bp-subdifferential} for the vector
\(V_p(x,y)\), which leads to the conclusion of \eqref{eq:FY-Busemann-characterization}.

It remains to establish the particular characterization
\eqref{eq:Fermat-rule-Bp}. To this end, first note that from \eqref{eq:Busemann-coupling}, we have $
    H_p(z,p)=0,$
    $z\in M$ and, hence,
$$
    f_p^*(p)
    =
    \sup_{z\in M}\bigl\{H_p(z,p)-f(z)\bigr\}
    =
    -\inf_{z\in M}f(z).
$$
Moreover, \eqref{eq:Vpxy} gives \(V_p(x,p)=0_x\). Applying
\eqref{eq:FY-Busemann-characterization} with \(y=p\), we obtain
$$
\begin{aligned}
    0_x\in\partial B^p f(x)
    \quad\Longleftrightarrow\quad
    f(x)+f_p^*(p)=H_p(x,p)
    \quad\Longleftrightarrow\quad
    f(x)-\inf_{z\in M}f(z)=0.
\end{aligned}
$$
Since \(x\in\operatorname{dom}f\), the last equality holds if and only if
\(x\in\operatorname*{argmin}_{z\in M}f(z)\), which completes the proof.
\end{proof}

\begin{Remark}\label{Rem-1001}
The preceding theorem does not require the coupling \(H_p\) to be
symmetric, because its proof never interchanges the two arguments of
\(H_p\). The key point is the identity
$$
    Y_p\bigl(x,V_p(x,y)\bigr)=y,
$$
which gives
$$
\begin{aligned}
&H_p\bigl(z,Y_p(x,V_p(x,y))\bigr)
-H_p\bigl(x,Y_p(x,V_p(x,y))\bigr)
=H_p(z,y)-H_p(x,y).
\end{aligned}
$$
Thus, the support increment has precisely the same orientation
\(H_p(\cdot,y)\) as both the section \(H_y=H_p(\cdot,y)\) and the
conjugate
$$
    f_p^*(y)
    =
    \sup_{z\in M}\bigl\{H_p(z,y)-f(z)\bigr\}.
$$
This agreement of orientations allows
Theorem~\ref{thm:FY-Busemann-characterization} to characterize the vanishing of the Fenchel-Young gap in terms of membership in the \(B^p\)-subdifferential.
By contrast, at the base point, a fixed-direction Busemann support is
represented by \(H_p(y,z)\), with the arguments reversed. Since
\(H_p(z,y)\) and \(H_p(y,z)\) need not coincide, the corresponding
support conditions are generally distinct. Indeed,
Example~\ref{ex:comparison-hyperbolic} in the next section provides a vector
\(v\in\partial^b f(0)\) such that \(v\notin\partial B^0 f(0)\).
Therefore, for \(y=Y_0(0,v)\),
Theorem~\ref{thm:FY-Busemann-characterization} implies that the
corresponding Fenchel-Young gap does not vanish, even though \(v\) is a fixed-direction Busemann subgradient. Consequently, fixed-direction Busemann subdifferentials do not, in general, characterize the zero set of the Fenchel--Young gap associated with \(f_p^*\).
\end{Remark}

\noindent\textbf{Behavior under isometries.}
Given $I:M\to M$ an isometry satisfying $I(q)=p$,  \cite{bento2023fenchel}, establishes the following transformation rule for conjugate functions associated with different base points: \begin{equation}\label{conjugateIso} f_p^*\circ I=(f\circ I)_q^*. \end{equation} By combining Theorem~\ref{thm:FY-Busemann-characterization} with \eqref{conjugateIso}, the next theorem describes how the corresponding subdifferentials transform under isometries of the Riemannian manifold. 
\begin{Theorem}\label{thm:Bp-isometry}
Let $I:M\to M$ be an isometry such that $I(q)=p$. Then
\[
    u\in\partial B^q(f\circ I)(b)
    \quad\Longleftrightarrow\quad
    DI_b(u)\in\partial B^p f(I(b)).
\]
\end{Theorem}

\begin{proof}
Suppose that $u\in\partial B^q(f\circ I)(b)$, and set
\(
    y:=Y_q(b,u).
\)
By Theorem~\ref{thm:FY-Busemann-characterization},
\[
    (f\circ I)(b)+(f\circ I)_q^*(y)=H_q(b,y).
\]
Since $I$ preserves distances, geodesics, and asymptoticity,
\[
    H_q(b,y)=H_p(I(b),I(y))
\]
and
\[
    I(y)=Y_p\bigl(I(b),DI_b(u)\bigr).
\]
Combining these identities with \eqref{conjugateIso}, we obtain
\[
    f(I(b))+f_p^*(I(y))=H_p(I(b),I(y)).
\]
Theorem~\ref{thm:FY-Busemann-characterization} now gives
\[
    DI_b(u)\in\partial B^p f(I(b)).
\]
The converse follows by applying the same argument to the inverse isometry
$I^{-1}$.
\end{proof}

\noindent\textbf{Busemann subdifferential of radial functions.}
\label{subsec:Bp-radial}
A function $f:M\to\mathbb{R}\cup\{+\infty\}$ is radial at $p$ if there
exists $h:[0,+\infty)\to\mathbb{R}\cup\{+\infty\}$ such that
\(
    f(x)=h(d(x,p)).
\)
As established in \cite{bento2023fenchel},
\[
    f_p^*=h^*\circ d(\cdot,p),
\]
where the conjugate of $h$ is taken over $[0,+\infty)$. Moreover, if $h$ is
proper, lower semicontinuous, and convex, then $f_p^{**}=f$.

As another application of Theorem~\ref{thm:FY-Busemann-characterization}, our objective now is to show that, away from the center of a differentiable
radial function, the Busemann subdifferential reduces to its Riemannian gradient. 

\begin{Theorem}\label{thm:Bp-subdifferential-radial}
Let $f:M\to\mathbb{R}\cup\{+\infty\}$ be radial at $p\in M$, with
\(
    f=h\bigl(d(\cdot,p)\bigr),
\)
where $h:[0,\infty)\to\mathbb{R}$ is differentiable and convex. Assume
that
\[
    h(0)=h'(0)=0
    \qquad\text{and}\qquad
    h'(t)>0,\qquad t>0,
\]
where $h'(0)$ denotes the right derivative at the origin. Then,
\(
    \partial B^p f(x)=\{\operatorname{grad}f(x)\}\), for
 $x\neq p$, and
\(
    \partial B^p f(p)=\{0\}.
\)
\end{Theorem}
\begin{proof}
Fix $x\neq p$ and set
\[
    r:=d(x,p),
    \qquad
    e_x:=\frac{\exp_p^{-1}x}{r}\in T_pM,
    \qquad
    w_x:=-\frac{\exp_x^{-1}p}{r}\in T_xM.
\]
It is known that for  $x\in M\setminus\{p\}$, one has
\begin{equation}\label{eq:gradient-radial-function}
    \operatorname{grad}f(x)
    =h'(r)w_x
    =-\frac{h'(r)}{r}\exp_x^{-1}p.
\end{equation}
Letting $s:=h'(r)>0$,
    $\eta_x:=-e_x$,
    $y_x:=\exp_p(-s\eta_x)=\exp_p(se_x)$,  note that the ray starting from $p$ in the direction $\eta_x$ and the ray starting from $x$ in the direction $-w_x$ lie on the same complete geodesic and are therefore asymptotic.
    Thus,  \eqref{eq:gradient-radial-function} yields \[ \operatorname{grad}f(x)=s w_x, \qquad \bigl\|\operatorname{grad}f(x)\bigr\|=s, \qquad -\frac{\operatorname{grad}f(x)} {\|\operatorname{grad}f(x)\|} =-w_x, \] and, by construction, $\eta_x$ is precisely the unit vector appearing in the definition of $Y_p\bigl(x,\operatorname{grad}f(x)\bigr)$. Consequently, 
    \begin{equation}\label{eq:1001}  Y_p\bigl(x,\operatorname{grad}f(x)\bigr) = \exp_p\left( -\bigl\|\operatorname{grad}f(x)\bigr\|\eta_x \right) =\exp_p(-s\eta_x) =y_x.
    \end{equation}
    On the other hand, because each Busemann function is $1$-Lipschitz and vanishes at its base
point,
\[
    B^p_{-\exp_p^{-1}u}(y_x)
    \leq d(p,y_x)=s,
\]
whenever $u\neq p$; the corresponding inequality for $H_p$ is also
valid when $u=p$, because $H_p(p,y_x)=0$. Hence,
\begin{equation}\label{eq:1002}
    H_p(u,y_x)\leq s\,d(u,p).
\end{equation}
Furthermore, since 
\(
    d(p,y_x)=s\)
and
  \(  B^p_{-e_x}(y_x)=s,
\)
it follows that
\begin{equation}\label{eq:radial-H-at-x}
    H_p(x,y_x)
    =rB^p_{-\exp_p^{-1}x}(y_x)
    =rs.
\end{equation}
Since $s=h'(r)$, the convexity of $h$ yields
\[
\begin{aligned}
    f(u)
    =h\bigl(d(u,p)\bigr)
    \geq h(r)+s\bigl(d(u,p)-r\bigr),\qquad u\in M.
\end{aligned}
\]
Combining this inequality with \eqref{eq:1001}, \eqref{eq:1002} and \eqref{eq:radial-H-at-x}, we obtain
\[ 
f(u)\geq f(x) +H_p\left( u,Y_p\bigl(x,\operatorname{grad}f(x)\bigr) \right) -H_p\left( x,Y_p\bigl(x,\operatorname{grad}f(x)\bigr) \right), \qquad u\in M. 
\] 
This is exactly the global support inequality in Definition~\ref{def:Bp-subdifferential} from which it follows that:
\begin{equation}\label{eq:gradient-in-Bp-subdifferential} \operatorname{grad}f(x)\in\partial B^p f(x). \end{equation}
 For the proof of reverse inclusion, let $v\in\partial B^p f(x)$ and
write:
\[
    y:=Y_p(x,v),
    \qquad
    \rho:=d(p,y)=\lVert v\rVert.
\]
The equality characterization in Theorem~\ref{thm:FY-Busemann-characterization}, together
with the radial conjugacy formula $f_p^*=h^*\circ d(\cdot,p)$, yields
\begin{equation}\label{eq:FY-radial-equality}
    h(r)+h^*(\rho)=H_p(x,y)=rB^p_{-e_x}(y)
    \leq r\rho,
\end{equation}
where the last inequality above follows from the $1$-Lipschitz property of the Busemann function.
On the other hand, the classical Fenchel-Young inequality for $h$
yields
\[
    h(r)+h^*(\rho)\geq r\rho,
\]
which, combined with \eqref{eq:FY-radial-equality},  implies
\[
    h(r)+h^*(\rho)=r\rho
    \quad\text{and}\quad
    B^p_{-e_x}(y)=\rho.
\]
In the one-dimensional case, the equality in the one-dimensional Fenchel-Young inequality gives
\[
    \rho\in\partial h(r)=\{h'(r)\},
\]
and, therefore, $\rho=h'(r)=s$.

We now recall the equality case of the Busemann estimate used above. For a
unit vector $e\in T_pM$ and $q\in M$,
\begin{equation}\label{eq:equality-case-Busemann-estimate}
    B^p_{-e}(q)=d(p,q)
    \quad\Longleftrightarrow\quad
    q=\exp_p\bigl(d(p,q)e\bigr).
\end{equation}
Indeed, the equivalence is immediate when $q=p$. Assume, therefore, that
$q\neq p$, set $L:=d(p,q)>0$, and let
$\sigma:[0,L]\to M$ be the unit-speed minimizing geodesic joining $p$
to $q$. Thus,
\[
    \sigma(0)=p,\qquad
    \sigma(L)=q,\qquad
    \lVert\dot\sigma(t)\rVert=1.
\]
Let us suppose first that
\(
    B^p_{-e}(q)=d(p,q)=L.
\)
Because $B^p_{-e}(p)=0$, it is easy to see that 
\[
\begin{aligned}
    L
    =B^p_{-e}(q)-B^p_{-e}(p)
    =\int_0^L
        \frac{d}{dt}B^p_{-e}(\sigma(t))\,dt
    =\int_0^L
        \bigl\langle
            \operatorname{grad}B^p_{-e}(\sigma(t)),
            \dot\sigma(t)
        \bigr\rangle\,dt.
\end{aligned}
\]
On the other hand, since $B^p_{-e}$ is $1$-Lipschitz and
\[
    \bigl\lVert
        \operatorname{grad}B^p_{-e}(\sigma(t))
    \bigr\rVert=1,
    \qquad
    \lVert\dot\sigma(t)\rVert=1,
\]
the Cauchy--Schwarz inequality yields
\[
    \bigl\langle
        \operatorname{grad}B^p_{-e}(\sigma(t)),
        \dot\sigma(t)
    \bigr\rangle
    \leq1,
    \qquad t\in[0,L].
\]
The integrand is continuous and bounded above by $1$, while its
integral over $[0,L]$ is equal to $L$. It must therefore be identically
equal to $1$; otherwise, continuity would imply that it is strictly
smaller than $1$ on a nontrivial subinterval, making the integral
strictly smaller than $L$. Consequently,
\[
    \bigl\langle
        \operatorname{grad}B^p_{-e}(\sigma(t)),
        \dot\sigma(t)
    \bigr\rangle=1,
    \qquad t\in[0,L].
\]
Equality in the Cauchy--Schwarz inequality, together with the fact that
both vectors have unit norm, implies that
\[
    \operatorname{grad}B^p_{-e}(\sigma(t))
    =\dot\sigma(t),
    \qquad t\in[0,L].
\]
In particular, at $t=0$,
\[
    \dot\sigma(0)
    =\operatorname{grad}B^p_{-e}(p)
    =e.
\]
By uniqueness of geodesics with prescribed initial data,
\[
    \sigma(t)=\exp_p(te),
    \qquad t\in[0,L],
\]
and, hence
\[
    q=\sigma(L)=\exp_p(Le)
    =\exp_p\bigl(d(p,q)e\bigr).
\]
Conversely, let us suppose now that
\[
    q=\exp_p\bigl(d(p,q)e\bigr).
\]
The ray defining $B^p_{-e}$ is $t\mapsto\exp_p(-te)$. Since these two
points lie on opposite sides of the same complete geodesic,
\[
    d\left(\exp_p\bigl(d(p,q)e\bigr),\exp_p(-te)\right)
    =d(p,q)+t.
\]
It follows directly from the definition of the Busemann function that
\[
\begin{aligned}
    B^p_{-e}(q)
    =\lim_{t\to\infty}
        \left[
            d\left(q,\exp_p(-te)\right)-t
        \right]
    =\lim_{t\to\infty}\bigl[d(p,q)+t-t\bigr]
    =d(p,q),
\end{aligned}
\]
which concludes the proof of \eqref{eq:equality-case-Busemann-estimate}.

To complete the reverse inclusion, we now apply \eqref{eq:equality-case-Busemann-estimate} with $e=e_x$, which allows us to obtain
\[
    y=\exp_p(\rho e_x)
    =\exp_p\bigl(h'(r)e_x\bigr)
    =y_x.
\]
The uniqueness of the asymptotic direction in
Definition~\ref{def:Bp-subdifferential} now yields
\[
    v=h'(r)w_x=\operatorname{grad}f(x),
\]
which, combined with \eqref{eq:gradient-in-Bp-subdifferential}, allows us to complete the reverse inclusion and, consequently, proves
\[
    \partial B^p f(x)=\{\operatorname{grad}f(x)\},
    \qquad x\neq p.
\]

It remains to consider the center $p$. Since $h(0)=0$ and $h'(t)>0$
for $t>0$, the point $p$ is the unique global minimizer of $f$.
Therefore,
\[
    0\in\partial B^p f(p).
\]
Conversely, take $v\in\partial B^p f(p)$ and suppose that $v\neq0$.
Set
\[
    \rho:=\lVert v\rVert,
    \qquad
    e:=\frac{v}{\rho},
    \qquad
    y:=\exp_p(\rho e).
\]
For $t>0$, let $z_t:=\exp_p(te)$. From the definition of $H_p$ and the
identity $B^p_{-e}(y)=\rho$, we have
\[
    H_p(z_t,y)=t\rho.
\]
The subgradient inequality at $p$ therefore gives
\[
    h(t)=f(z_t)
    \geq f(p)+H_p(z_t,y)-H_p(p,y)
    =t\rho,
    \qquad t>0.
\]
Dividing by $t$ and letting $t$ go to $0$, we conclude that
\[
    \rho\leq h'(0)=0,
\]
which contradicts $\rho>0$. Thus $v=0$, and consequently
\(
    \partial B^p f(p)=\{0\},
\)
which concludes the proof of the theorem.
\end{proof}

\medskip
\begin{Remark}[Necessity of the assumption $h'(0)=0$]
The condition $h'(0)=0$ cannot, in general, be omitted. Indeed, let
$M=\mathbb{R}^n$, $p=0$, and consider
\[
    h(t)=t,\qquad t\geq0.
\]
In this case, $h$ is differentiable, convex, satisfies $h(0)=0$,
$h'(t)=1>0$ for every $t\geq0$, and the associated radial function is
\[
    f(x)=h(\lVert x\rVert)=\lVert x\rVert.
\]
Since the $B^0$-subdifferential coincides with the classical convex
subdifferential in the Euclidean setting, we have
\[
    \partial B^0 f(0)
    =\partial\lVert\cdot\rVert(0)
    =\bigl\{v\in\mathbb{R}^n:\lVert v\rVert\leq1\bigr\},
\]
which is not the singleton $\{0\}$. More generally, if $h$ is convex
and right differentiable at the origin, then
\[
    \partial B^p f(p)
    =\bigl\{v\in T_pM:\lVert v\rVert\leq h'_+(0)\bigr\}.
\]
Consequently, within this class of radial functions, the equality
$\partial B^p f(p)=\{0\}$ holds precisely when $h'_+(0)=0$.
\end{Remark}

\subsection{Comparison with fixed-direction Busemann subdifferentials}
\label{subsec:comparison-Busemann-subdifferentials}

During the development of this research program, other notions of
subdifferential based on Busemann functions were proposed; see
\cite{goodwin2024subgradient,criscitiello2025horospherically,ferreira2026subdifferential}.
Accounting for differences in sign conventions, normalizations, and
parametrizations, \cite[Remarks~2 and~3]{ferreira2026subdifferential} shows that the Busemann
subdifferential of \cite{goodwin2024subgradient}, the horospherical
subdifferential of \cite{criscitiello2025horospherically}, and the
$B$-subdifferential of \cite{ferreira2026subdifferential} are equivalent at the level of their
underlying global support inequality. It is therefore sufficient, for the
purpose of the present comparison, to use the unified formulation given in
\cite{ferreira2026subdifferential}.
More precisely, for a convex function $f:M\to\mathbb{R}$ and $x\in M$, the
$B$-subdifferential in \cite{ferreira2026subdifferential} is given, in our notation, by
\begin{equation}\label{eq:fixed-B-subdifferential}
    s\in\partial^b f(x)
    \quad\Longleftrightarrow\quad
    f(z)\geq f(x)+\|s\|B^x_{-s}(z),
    \qquad z\in M.
\end{equation}
The right-hand side of \eqref{eq:fixed-B-subdifferential} is a single scaled
Busemann function of $z$, associated with the fixed ideal direction determined
by $-s$. By contrast, the supporting term in
Definition~\ref{def:Bp-subdifferential} is the difference
\begin{equation}\label{eq:variable-Busemann-support}
\begin{aligned}
&H_p\bigl(z,Y_p(x,v)\bigr)-H_p\bigl(x,Y_p(x,v)\bigr)\\
&\quad=
\left[
d(p,z)
B^p_{-\frac{\exp_p^{-1}z}{d(p,z)}}
   \bigl(Y_p(x,v)\bigr)
\right]
-
\left[
d(p,x)
B^p_{-\frac{\exp_p^{-1}x}{d(p,x)}}
   \bigl(Y_p(x,v)\bigr)
\right],
\end{aligned}
\end{equation}
whenever $x,z\neq p$, with the corresponding weighted Busemann term defined to
be zero when its first argument is $p$. In particular, both the weight and the
Busemann direction in the first term of
\eqref{eq:variable-Busemann-support} depend on $z$. Moreover, the associated
Busemann function is evaluated at the dual point $Y_p(x,v)$, rather than at
$z$. Hence, unlike \eqref{eq:fixed-B-subdifferential}, the support model in
\eqref{eq:variable-Busemann-support} is not, in general, generated by a single
Busemann function of $z$ with a fixed ideal direction.

The distinction between the two constructions disappears in the flat
setting. Indeed, let $M=\mathbb{R}^n$, fix $p\in M$, and let
$f:M\to\mathbb{R}\cup\{+\infty\}$ be proper and convex. For any
$x\in\operatorname{dom}f$ and $v\in T_xM\simeq\mathbb{R}^n$, we have
\[
    Y_p(x,v)=p+v.
\]
Consequently, by \eqref{acoplamento1},
\[
\begin{aligned}
&H_p\bigl(z,Y_p(x,v)\bigr)
    -H_p\bigl(x,Y_p(x,v)\bigr)\\
&\qquad
=\langle v,z-p\rangle-\langle v,x-p\rangle
=\langle v,z-x\rangle,
\qquad z\in M.
\end{aligned}
\]
For $v\neq 0$, the corresponding Euclidean Busemann function also
satisfies
\[
    \|v\|B^x_{-v}(z)=\langle v,z-x\rangle.
\]
When $v=0$, both expressions are understood to vanish identically.
Hence, in Euclidean space, the support inequalities defining
$\partial B^p f(x)$ and $\partial^b f(x)$ both reduce to the classical
subgradient inequality
\[
    f(z)\geq f(x)+\langle v,z-x\rangle,
    \qquad z\in M.
\]
It follows that the two Busemann-based subdifferentials coincide with
the classical convex subdifferential:
\begin{equation}\label{equality-Flatcase}
        \partial B^p f(x)=\partial^b f(x)=\partial f(x),
    \qquad x\in\operatorname{dom}f.
\end{equation}

The following example shows that the coincidence in \eqref{equality-Flatcase} does not extend to
general Hadamard manifolds.
\begin{Example}
\label{ex:comparison-hyperbolic}
Consider the Poincaré disk model \(\mathbb D\) of the hyperbolic plane
with constant sectional curvature \(-1\), and let \(p=0\). Let
\(e_1,e_2\in T_0\mathbb D\) be orthogonal unit vectors, let \(r>0\),
and set
\[
    v:=re_1,
    \qquad
    y:=\exp_0(v),
    \qquad
    f(z):=H_0(y,z)=rB^0_{-e_1}(z).
\]
Because \(f\) is a smooth geodesically convex function and
\(
    f(z)-f(0)=\|v\|B^0_{-v}(z),
\)
it follows from the fixed-direction support condition
\eqref{eq:fixed-B-subdifferential} that $v\in\partial^b f(0)$.
Now take $0<b<r$ and set $z:=\exp_0(be_2)$.
Since \(Y_0(0,v)=y\), the support inequality defining
\(\partial B^0 f(0)\) would require

$$
    f(z)-f(0)
    \geq
    H_0(z,y)-H_0(0,y).
$$
Applying Proposition~\ref{prop:Hp-asymmetry-phi} with
$$
    \eta=-e_1,
    \qquad
    \theta=e_2,
    \qquad
    s=r,
$$
and with the radial parameter in the direction \(\theta\) equal to
\(b\), we obtain
\begin{equation*}
\label{eq:comparison-H-values-phi}
H_0(y,z)
=
r\log(\cosh b)
=
rb\,\varphi(b),
\qquad
H_0(z,y)
=
b\log(\cosh r)
=
rb\,\varphi(r),
\end{equation*}
where $\varphi(t)=\log(\cosh t)/t,$ $t>0$, is defined in \eqref{eq:functionAux}. 
Since \(f(z)=H_0(y,z)\), \(f(0)=H_0(y,0)=0\), and
\(H_0(0,y)=0\), it follows that

$$
    f(z)-f(0)
    =
    rb\,\varphi(b),\qquad  H_0(z,y)-H_0(0,y)
    =
    rb\,\varphi(r) 
$$
Because \(\varphi\) is strictly increasing and \(0<b<r\), one has

$$
    rb\,\varphi(b)
    <
    rb\,\varphi(r).
$$
Equivalently,
$$
    f(z)-f(0)
    =
    H_0(y,z)
    <
    H_0(z,y)
    =
    H_0(z,y)-H_0(0,y).
$$
Therefore, the support inequality
\eqref{eq:Bp-support} fails at \(z\), and consequently
$
    v\notin\partial B^0 f(0).
$
Thus, the failure of inclusion is a direct consequence of the
nonreciprocity of the Busemann coupling:

$$
    H_0(z,y)-H_0(y,z)
    =
    rb\bigl[\varphi(r)-\varphi(b)\bigr]
    >0.
$$
\end{Example}
The preceding example shows that the fixed-direction Busemann
subdifferential need not be contained in the \(B^p\)-subdifferential.
We now establish the reverse separation in the same Poincaré disk.
\begin{Example}[Reverse separation in the Poincaré disk]
\label{ex:reverse-comparison-hyperbolic}
Consider the same Poincaré disk model \(\mathbb D\) and retain the
notation of Example~\ref{ex:comparison-hyperbolic}. Thus, \(p=0\),
\(e_1,e_2\in T_0\mathbb D\) are orthogonal unit vectors, \(r>0\),
$
    v:=re_1,$   
    $y:=\exp_0(v).
$
Let
    $\ell=\{\exp_0(te_2):t\in\mathbb R\}$
be the complete hyperbolic geodesic through \(0\) orthogonal to the
direction of \(v\). Define \(g:\mathbb D\to\mathbb R\) by
\begin{equation}
\label{eq:reverse-separation-function-phi}
g(z)
:=
r\varphi(r)\,d(0,z)
+
\tanh(r)\,d(z,\ell),
\end{equation}
where $\varphi(t)=(\log\cosh t)/t$. 
We claim that \(g\) is finite-valued, geodesically convex, and $$
    v\in\partial B^0 g(0)
    \qquad\text{but}\qquad
    v\notin\partial^b g(0).
$$
First, observe that \(g\) is finite-valued and geodesically convex. Indeed, \(\log(\cosh r)>0\) and \(\tanh(r)>0\), while \(d(0,\cdot)\) and \(d(\cdot,\ell)\) are geodesically convex because they are distance functions to nonempty closed geodesically convex subsets of a Hadamard manifold. Moreover, since \(0\in\ell\), we have \(g(0)=0\).
Now, let us prove that \(v\in\partial B^0g(0)\). From the definition of the
primal-dual correspondence,
$$
    Y_0(0,v)=\exp_0(v)=y.
$$
Let \(z\in\mathbb D\setminus\{0\}\), and set
$t:=d(0,z)$ and 
$u:=(\exp_0^{-1}z)/t.
$
Taking into account that $y=\exp_0(re_1)=\exp_0\bigl(-r(-e_1)\bigr),
$ letting \(\alpha\) be the angle between \(e_1\) and  \(-u\),  Lemma~\ref{lem:Busemann-constant-curvature} yields
\begin{equation}
\label{eq:H0-general-reverse-separation}
H_0(z,y)
=
tB^0_{-u}(y)
=
t\log\bigl(\cosh r-\sinh r\cos\alpha\bigr).
\end{equation}
Let \(q:=P_\ell(z)\) be the metric projection of \(z\) onto \(\ell\),
and set
$$
    \delta:=d(z,\ell)=d(z,q).
$$
Since \(\ell\) is a complete geodesic, the geodesic segment joining
\(q\) to \(z\) is orthogonal to \(\ell\). Hence, the geodesic triangle
with vertices \(0\), \(q\), and \(z\) is right-angled at \(q\), with
hypotenuse of length \(t\).
Let \(\beta\in[0,\pi/2]\) be the acute angle between \(u\) and ${span}\{e_2\}$ (space generated by $e_2$).
Since \(e_1\perp e_2\), we have
$$
    \sin\beta
    =
    |\langle u,e_1\rangle|
    =
    |\cos\alpha|.
$$
The classical hyperbolic law of sines applied to the right triangle
\boxed{(0,q,z)} gives
$$
    \frac{\sinh\delta}{\sin\beta}
    =
    \frac{\sinh t}{\sin(\pi/2)}
    =
    \sinh t,
$$
and, consequently,
\begin{equation}
\label{eq:distance-to-geodesic-reverse-separation}
\sinh\delta
=
\sinh t\;|\cos\alpha|.
\end{equation}
This identity also holds in the degenerate cases
\(\cos\alpha=0\) and \(|\cos\alpha|=1\).
If \(\cos\alpha>0\), then
$
    \log\bigl(\cosh r-\sinh r\,\cos\alpha\bigr)
    \leq
    \log(\cosh r)
    =
    r\varphi(r)
$
and, therefore, 
$$
H_0(z,y)
    \leq
    tr\varphi(r)
    \leq
    g(z).
$$
Let us suppose now that \(\cos\alpha\leq0\) and note that \[\log\bigl(\cosh r-\sinh r\cos\alpha\bigr)
=
\log(\cosh r)
+
\log\bigl(1-\tanh(r)\cos\alpha\bigr).\]
Hence, using that   \(\log(1+s)\leq s\) for \(s\geq0\), definition of $\varphi(\cdot)$ and taking into account that $\cos\alpha\leq 0$,
\begin{align*}
\log\bigl(\cosh r-\sinh r\cos\alpha\bigr)
\leq
r\varphi(r)-\tanh(r)\cos\alpha.
\end{align*}
Since \(\cos\alpha\in[-1,0]\), the convexity of the hyperbolic $\sinh$ and
\(\sinh(0)=0\) imply

$$
    \sinh(-t\cos\alpha)\leq -\cos\alpha\sinh t.
$$
Combining this inequality with
\eqref{eq:distance-to-geodesic-reverse-separation} and using the
monotonicity of \(\emph{arcsinh}(\cdot)\), we obtain

$$
    -t\cos\alpha
    \leq
    \operatorname{arsinh}(-\cos\alpha\sinh t)
    =
    \delta.
$$
Consequently,
\begin{align*}
H_0(z,y)
\leq
tr\varphi(r)-t\tanh(r)\cos\alpha
\leq
tr\varphi(r)+\tanh(r)\delta
=g(z).
\end{align*}
For \(z=0\), the same inequality follows from
\(H_0(0,y)=g(0)=0\). Thus,
$$
    g(z)
    \geq
    g(0)+H_0(z,y)-H_0(0,y),
    \qquad z\in\mathbb D.
$$
Since \(y=Y_0(0,v)\), Definition~\ref{def:Bp-subdifferential} yields
$$
    v\in\partial B^0g(0).
$$
We next express the fixed-direction support in terms of the same
coupling. Since \(y=\exp_0(re_1)\), one has
\begin{equation}
\label{eq:fixed-support-as-H-reverse}
H_0(y,z)
=
rB^0_{-e_1}(z)
=
\|v\|B^0_{-v}(z),
\qquad z\in\mathbb D.
\end{equation}
Therefore, membership \(v\in\partial^b g(0)\) would require

$$
    g(z)-g(0)\geq H_0(y,z),
    \qquad z\in\mathbb D.
$$
Choose \(T>r\) and set
$$
    z_T:=\exp_0(Te_2)\in\ell.
$$
Applying Proposition~\ref{prop:Hp-asymmetry-phi} with

$$
    \eta=-e_1,
    \qquad
    \theta=e_2,
    \qquad
    s=r,
$$
and with the radial parameter in the direction \(\theta\) equal to
\(T\), we obtain
\begin{equation}
\label{eq:reverse-H-values-phi}
H_0(z_T,y)
=
rT\varphi(r),
\qquad
H_0(y,z_T)
=
rT\varphi(T).
\end{equation}
Since \(z_T\in\ell\), definition
\eqref{eq:reverse-separation-function-phi} gives

$$
    g(z_T)
    =
    rT\varphi(r)
    =
    H_0(z_T,y).
$$
Because \(T>r\) and \(\varphi\) is strictly increasing, it follows
from \eqref{eq:reverse-H-values-phi} that

$$
    g(z_T)-g(0)
    =
    H_0(z_T,y)
    =
    rT\varphi(r)
    <
    rT\varphi(T)
    =
    H_0(y,z_T).
$$
In view of \eqref{eq:fixed-support-as-H-reverse}, this is precisely
$$
    g(z_T)
    <
    g(0)+\|v\|B^0_{-v}(z_T).
$$
Hence, the fixed-direction support inequality fails at \(z_T\), and
therefore

$$
    v\notin\partial^b g(0).
$$
\end{Example}

\begin{Remark}
Examples~\ref{ex:comparison-hyperbolic} and
\ref{ex:reverse-comparison-hyperbolic} establish that
$$
    \partial^b f(0)\nsubseteq\partial B^0f(0)
    \qquad\text{and}\qquad
    \partial B^0g(0)\nsubseteq\partial^b g(0).
$$
In the first example, the separation follows from
$$
    H_0(z,y)-H_0(y,z)
    =
    rb\bigl[\varphi(r)-\varphi(b)\bigr]>0,
    \qquad 0<b<r,
$$

whereas in the second it follows from
$$
    H_0(y,z_T)-H_0(z_T,y)
    =
    rT\bigl[\varphi(T)-\varphi(r)\bigr]>0,
    \qquad T>r.
$$
Thus, the two failures of inclusion correspond to opposite radial
manifestations of the same nonreciprocity of \(H_0\). Consequently,
although the two Busemann-based subdifferentials coincide in the
Euclidean setting, they are genuinely incomparable on the Poincaré disk.
\end{Remark}

The preceding examples show that the order of the arguments of \(H_p\)
gives rise to two genuinely distinct support theories. In particular,
Example~\ref{ex:comparison-hyperbolic} should not be interpreted as showing that the
\(B^p\)-subdifferential fails on its natural generating sections. Indeed,
the function considered there is a section of the reversed form
$$
    z\longmapsto H_p(y,z),
$$
which generates the fixed-direction Busemann support at the base point,
whereas both the \(B^p\)-subdifferential and the conjugate \(f_p^*\) are
built from sections with the orientation

$$
    z\longmapsto H_p(z,y).
$$
By Lemma~\ref{lem:Y-V-inverse-H-support}, each such section admits an
exact global \(B^p\)-support at every point. The following proposition
extends this observation to upper envelopes of these natural generators:
whenever one of the generating sections is active at a point, it
determines a \(B^p\)-subgradient there. Thus, the separation exhibited by
the preceding examples reflects the asymmetry of \(H_p\) and the
resulting distinction between two support theories, rather than a
failure of the present subdifferential on its own generating class.
The preceding discussion suggests a natural functional class associated
with the orientation \(H_p(\cdot,y)\). We now show that, for
finite-valued functions, pointwise nonemptiness of the
\(B^p\)-subdifferential is equivalent to representability as an
\(H_p\)-envelope whose defining supremum is attained at every point.

\begin{definition}[Pointwise-attained \(H_p\)-envelope]
\label{def:pointwise-attained-Hp-envelope}
A function \(f:M\to\mathbb R\) is called an \(H_p\)-envelope if there
exist a nonempty index set \(A\), points \(y_\alpha\in M\), and constants
\(a_\alpha\in\mathbb R\), \(\alpha\in A\), such that
$$
    f(z)
    =
    \sup_{\alpha\in A}
    \bigl\{H_p(z,y_\alpha)-a_\alpha\bigr\},
    \qquad z\in M.
$$
For \(x\in M\), define the active index set
$$
    A_f(x)
    :=
    \left\{
        \alpha\in A:
        f(x)=H_p(x,y_\alpha)-a_\alpha
    \right\}.
$$
The envelope is said to be pointwise attained if
\(A_f(x)\neq\varnothing\) for every \(x\in M\).
\end{definition}
\begin{Theorem}[Nonemptiness and pointwise-attained \(H_p\)-envelopes]
\label{thm:Hp-envelope-characterization}
Let \(f:M\to\mathbb R\). Then the following statements are equivalent:
\begin{enumerate}
\item[\rm (i)]
$
    \partial B^p f(x)\neq\varnothing$,
$x\in M;
$
\item[\rm (ii)] The function \(f\) is a pointwise-attained
\(H_p\)-envelope.
\end{enumerate}
Moreover, for any pointwise-attained representation of \(f\),
\begin{equation}
\label{eq:active-Hp-generators-subgradients}
\left\{
V_p(x,y_\alpha):
\alpha\in A_f(x)
\right\}
\subseteq
\partial B^p f(x),
\qquad x\in M.
\end{equation}
\end{Theorem}
\begin{proof}
Assume first that \({\rm(ii)}\) holds. Fix \(x\in M\) and choose
\(\alpha_x\in A_f(x)\). Since \(\alpha_x\) is active at \(x\),
$$  f(x)=H_p(x,y_{\alpha_x})-a_{\alpha_x}.
$$
For every \(z\in M\), the envelope representation gives
$$
\begin{aligned}
    f(z)
    \geq H_p(z,y_{\alpha_x})-a_{\alpha_x}
=f(x)+H_p(z,y_{\alpha_x})-H_p(x,y_{\alpha_x}).
\end{aligned}
$$
Using  $Y_p\bigl(x,V_p(x,y_{\alpha_x})\bigr)=y_{\alpha_x},
$
we conclude that $V_p(x,y_{\alpha_x})\in\partial B^p f(x).
$
Due to the arbitrariness of \(x\), \({\rm(i)}\) follows. The same argument
applied to every \(\alpha\in A_f(x)\) proves
\eqref{eq:active-Hp-generators-subgradients}.

Conversely, assume that \({\rm(i)}\) holds, and consider the index set
$$
A:=\left\{
        (\xi,v):
        \xi\in M,\;
        v\in\partial B^p f(\xi)
    \right\}.
$$
For \(\alpha=(\xi,v)\in A\), define
$y_\alpha:=Y_p(\xi,v)$ and $a_\alpha:=H_p(\xi,y_\alpha)-f(\xi).
$
Because \(v\in\partial B^p f(\xi)\), the defining support inequality
yields
$$
\begin{aligned}
    f(z)
    \geq
f(\xi)+H_p(z,y_\alpha)-H_p(\xi,y_\alpha)
    =H_p(z,y_\alpha)-a_\alpha,
    \qquad z\in M.
\end{aligned}
$$
Taking the supremum over \(\alpha\in A\), we obtain
$$
    f(z)
    \geq
    \sup_{\alpha\in A}
    \bigl\{H_p(z,y_\alpha)-a_\alpha\bigr\}.
$$
On the other hand, fix \(z\in M\). By \({\rm(i)}\), choose
\(v_z\in\partial B^p f(z)\) and let
\(\alpha_z:=(z,v_z)\). By construction,
$$
\begin{aligned}
    H_p(z,y_{\alpha_z})-a_{\alpha_z}
    =
    H_p(z,y_{\alpha_z})
    - \bigl[H_p(z,y_{\alpha_z})-f(z)\bigr]
    =f(z).
\end{aligned}
$$
Therefore,
$$
    f(z)
    =
    \sup_{\alpha\in A}
    \bigl\{H_p(z,y_\alpha)-a_\alpha\bigr\},
$$
and the supremum is attained at \(\alpha_z\). Hence \(f\) is a
pointwise-attained \(H_p\)-envelope, and the proof is complete.
\end{proof}

\begin{Corollary}[Finite \(H_p\)-envelopes]
\label{cor:finite-Hp-envelopes}
Let \(y_1,\ldots,y_m\in M\) and \(a_1,\ldots,a_m\in\mathbb R\), and
define
$$
    f(z)
    :=
    \max_{1\leq i\leq m}
    \bigl\{H_p(z,y_i)-a_i\bigr\}.
$$
Then,
$
    \partial B^p f(x)\neq\varnothing,$
    $x\in M$.
More precisely, every active index
$$
    i\in
    \operatorname*{argmax}_{1\leq j\leq m}
    \bigl\{H_p(x,y_j)-a_j\bigr\}
$$
determines a \(B^p\)-subgradient
$    V_p(x,y_i)\in\partial B^p f(x).
$
\end{Corollary}
The preceding characterization identifies pointwise-attained
\(H_p\)-envelopes as precisely the finite-valued functions admitting
\(B^p\)-supports at every point. The following two examples illustrate
how familiar models already considered in our development fit naturally
within this framework. The first shows that the radial convex functions
covered by Theorem~\ref{thm:Bp-subdifferential-radial} admit pointwise-attained
\(H_p\)-envelope representations, while the second provides an explicit
representation of the fundamental nonsmooth model
\(\lambda d(p,\cdot)\). Thus, their purpose is to exhibit concrete
smooth and nonsmooth instances of the functional class characterized
above.

\begin{Example}[Radial functions within the \(H_p\)-envelope framework]
\label{rem:radial-Hp-envelopes}
The radial functions considered in Theorem~\ref{thm:Bp-subdifferential-radial}
provide a concrete class of attained \(H_p\)-envelopes. Indeed, for
each \(\xi\in M\), set
\[
    v_\xi:=\operatorname{grad}f(\xi),\qquad
    y_\xi:=Y_p(\xi,v_\xi),\qquad
    a_\xi:=H_p(\xi,y_\xi)-f(\xi).
\]
Since \(v_\xi\in\partial B^p f(\xi)\), the corresponding support
inequality gives
\[
    f(z)\geq H_p(z,y_\xi)-a_\xi,
    \qquad z,\xi\in M.
\]
Consequently,
\[
    f(z)
    =
    \sup_{\xi\in M}
    \bigl\{H_p(z,y_\xi)-a_\xi\bigr\},
\]
where the supremum is attained at \(\xi=z\). Thus, the radial class
treated above is contained in the natural class generated by the
correctly oriented sections \(H_p(\cdot,y)\). Radiality alone, however,
does not imply such a representation without the hypotheses ensuring
the \(B^p\)-support established in Theorem~\ref{thm:Bp-subdifferential-radial}.
\end{Example}
\begin{Example}[Distance from the base point within the
\(H_p\)-envelope framework]
\label{ex:distance-Hp-envelope}
Let \(M\) be a Hadamard manifold, \(p\in M\), and 
\(\lambda>0\). We claim that
\begin{equation}
\label{eq:distance-Hp-envelope}
\lambda d(p,z)
=
\max_{\substack{y\in M\ d(p,y)=\lambda}}
H_p(z,y),
\qquad z\in M.
\end{equation}
Indeed, if \(z=p\), then both sides of
\eqref{eq:distance-Hp-envelope} vanish, since $H_p(p,y)=0,$ $y\in M.$
Now let \(z\neq p\), and set
$$
    r:=d(p,z),
    \qquad    e_z:=\frac{\exp_p^{-1}z}{r}.
$$
For every \(y\in M\) satisfying \(d(p,y)=\lambda\), the definition of
\(H_p\) gives
$H_p(z,y)=rB_{-e_z}^p(y).
$
Because \(B_{-e_z}^p\) is \(1\)-Lipschitz and
\(B_{-e_z}^p(p)=0\), we have
$$
    B_{-e_z}^p(y)
    =
    B_{-e_z}^p(y)-B_{-e_z}^p(p)
    \leq d(p,y)
    =
    \lambda.
$$
Consequently, $H_p(z,y)
    \leq
    \lambda r
    =
    \lambda d(p,z),$
and, hence,
$$
    \sup_{\substack{y\in M\\ d(p,y)=\lambda}}
    H_p(z,y)
    \leq
    \lambda d(p,z).
$$
To show that this upper bound is attained, define
$y_z:=\exp_p(\lambda e_z).
$
Then \(d(p,y_z)=\lambda\). From the definition of the Busemann
function,
\begin{equation}\label{eq:Busemann10001}
\begin{aligned}
    B_{-e_z}^p(y_z)
    &=
    \lim_{t\to+\infty}
    \left[
        d\bigl(y_z,\exp_p(-te_z)\bigr)-t
    \right].
\end{aligned}
\end{equation}
Since \(y_z=\exp_p(\lambda e_z)\) and
\(\exp_p(-te_z)\) lie on the same complete geodesic and geodesic segments are minimizing, one has $d\bigl(y_z,\exp_p(-te_z)\bigr)=\lambda+t$. Thus, from \eqref{eq:Busemann10001}
we obtain $B_{-e_z}^p(y_z)
    =\lambda,
$
and, hence $H_p(z,y_z)
    =
    rB_{-e_z}^p(y_z)
    =
    \lambda r
    =
    \lambda d(p,z)$, which 
proves \eqref{eq:distance-Hp-envelope}.
Thus, \(\lambda d(p,\cdot)\) is an attained \(H_p\)-envelope generated
by the sections \(H_p(\cdot,y)\) with \(d(p,y)=\lambda\). By
Theorem~\ref{thm:Hp-envelope-characterization},

$$
    \partial B^p\bigl(\lambda d(p,\cdot)\bigr)(z)
    \neq\varnothing,
    \qquad z\in M.
$$
It is worth pointing out that, for \(\lambda>0\), this example is not covered directly
by the differentiable radial case at \(p\), since \(d(p,\cdot)\) is
not differentiable at its center.
\end{Example}

The Examples~\ref{ex:comparison-hyperbolic} and
\ref{ex:reverse-comparison-hyperbolic} are two separation examples that establish the incomparability of the two
Busemann-based subdifferentials, but they do not by themselves determine
the variational meaning of their nonemptiness. The following theorem
reveals a sharper distinction at the distinguished base point. For
functions differentiable at \(p\) in hyperbolic space, nonemptiness of the
\(B^p\)-subdifferential is equivalent to global minimality and, by
Theorem~\ref{thm:FY-Busemann-characterization}, to the vanishing of the Fenchel-Young gap at the dual
point \(y=p\).

\begin{Theorem}[Base-point rigidity in hyperbolic space]
\label{thm:Bp-base-point-rigidity}
Let \(M=\mathbb H_k^n\), with \(n\geq2\) and constant sectional curvature
\(-k^2\), where \(k>0\), and let \(p\in M\). If
\(f:M\to\mathbb R\) is differentiable at \(p\), then
$$
    \partial B^p f(p)
    =
    \begin{cases}
        \{0_p\},
        &\displaystyle
        p\in\operatorname*{argmin}_{z\in M}f(z),\\[2mm]
        \varnothing,
        &\displaystyle
        p\notin\operatorname*{argmin}_{z\in M}f(z),
    \end{cases}
$$
where \(0_p\) denotes the zero vector in \(T_pM\).
\end{Theorem}
\begin{proof}
By the Fermat rule in Theorem~\ref{thm:FY-Busemann-characterization},
$0_p\in\partial B^p f(p)$ if and only if $p\in\operatorname*{argmin}_{z\in M}f(z).
$
It therefore remains to prove that \(\partial B^p f(p)\) cannot contain a
nonzero vector.
Suppose, to the contrary, that $v\in\partial B^p f(p)$, for some $v\neq0$ and set
$s:=\|v\|$, $y:=Y_p(p,v)=\exp_p(v)$.
Since \(n\geq2\), there exists a unit vector \(u\in T_pM\) orthogonal to
\(v\). For \(t>0\), define
$$
z_t^+:=\exp_p(tu),
    \qquad
    z_t^-:=\exp_p(-tu).
$$
The support inequality defining \(v\in\partial B^p f(p)\), together with
\(H_p(p,y)=0\), gives

$$
    f(z_t^\pm)-f(p)\geq H_p(z_t^\pm,y).
$$
Applying Proposition~\ref{prop:Hp-asymmetry-phi} to the orthogonal
directions \(u\) and \(v/\|v\|\), we obtain

$$
    H_p(z_t^\pm,y)
    =
    \frac{t}{k}\log\bigl(\cosh(ks)\bigr),
$$
and, consequently,
$   f(z_t^\pm)-f(p)
    \geq
    (t/k)\log\bigl(\cosh(ks)\bigr).
$
Because \(f\) is differentiable at \(p\), one has
$$
    f(z_t^+)-f(p)
    =
t\langle\operatorname{grad}f(p),u\rangle+o(t),\quad f(z_t^-)-f(p)
    =
    -t\langle\operatorname{grad}f(p),u\rangle+o(t).
$$
Dividing the preceding support inequalities by \(t\) and letting
\(t\downarrow0\) yields
$$
\langle\operatorname{grad}f(p),u\rangle
    \geq
    \frac{1}{k}\log\bigl(\cosh(ks)\bigr),\qquad  -\langle\operatorname{grad}f(p),u\rangle
    \geq
    \frac{1}{k}\log\bigl(\cosh(ks)\bigr).
$$
Adding these inequalities gives
$    0
    \geq
(2/k)\log\bigl(\cosh(ks)\bigr),$
which is impossible because \(s>0\). Hence,
$$
    \partial B^p f(p)\subseteq\{0_p\}.
$$
Combining this inclusion with the Fermat rule proves the result.
\end{proof}

\begin{Remark}[Nonemptiness, variational meaning, and negative curvature]
The preceding result should be interpreted as a rigidity property of the
\(B^p\)-subdifferential at its base point. In hyperbolic space, a
nonzero vector \(v\in T_pM\) produces strictly positive Busemann
increments along both orientations of every geodesic through \(p\)
orthogonal to \(v\). For a function differentiable at \(p\), the resulting
global support inequalities are incompatible with its first-order
behavior.
Combined with the Fenchel--Young characterization in
Theorem~\ref{thm:FY-Busemann-characterization}, this rigidity yields a
statement stronger than the Fermat rule alone: mere nonemptiness of the
\(B^p\)-subdifferential is equivalent both to global minimality and to the
vanishing of the Fenchel-Young gap at the base point. 
Thus, nonemptiness provides a sharp primal--dual certificate of global
minimality. This rigidity has no Euclidean analogue and illustrates how
negative curvature strengthens the global character of the Busemann
support condition.

Although Hamilton and Moitra
\cite{hamilton2021nogo} and Criscitiello and Boumal
\cite{criscitiello2022negative} have dealt with a different mathematical
question, they reflect the same general principle. In their setting,
negative curvature obstructs Euclidean-type acceleration for smooth,
strongly geodesically convex optimization. Here, it prevents a differentiable
function from admitting a nonzero global \(B^p\)-support at the base
point. Thus, Euclidean-type regularity assumptions do not by themselves
ensure that global structures of convex analysis and optimization retain
their Euclidean behavior in negatively curved spaces.
\end{Remark}

\section{Fenchel biconjugation and a measure of nonlinearity}
\label{Sec4}

The purpose of this section is to identify, in constant sectional curvature
\(-k^2\), the geometric quantity underlying two phenomena that have no
counterpart in Euclidean convex analysis. The first is the structural
difference, discussed in the preceding section, between our subdifferential
and the Busemann-based subdifferentials proposed in the literature. The second
is the possible failure of exact Fenchel biconjugation. Although these
phenomena concern different parts of the theory, both are governed by the
same lack of reciprocity of the Busemann coupling, namely, by the difference
\[
    H_p(z,y)-H_p(y,z).
\]
Our aim is to show that the maximal size of this asymmetry provides a common
\emph{measure of nonlinearity}: it both distinguishes the two types of
Busemann supports and quantifies the deviation of a function from its
biconjugate.

This perspective is motivated by the Fenchel-Moreau theorem, one of the
fundamental results of classical convex analysis. In finite-dimensional
Euclidean spaces, a proper function coincides with its biconjugate if and only
if it is lower semicontinuous and convex, and the classical inequality
\(f^{**}\leq f\) relies on the symmetry of the underlying bilinear pairing.
For a Hadamard manifold of constant sectional curvature \(-k^2\), the
Busemann coupling \(H_p\) is generally not symmetric. Consequently, the
relative order between a function and the biconjugate obtained by applying
the same conjugation twice cannot be inferred from the classical argument and
must instead be analyzed through the above measure of nonlinearity.

This phenomenon already appears for Busemann functions. More precisely, let
\(M\) have constant sectional curvature \(-k^2\), with \(k>0\), and let
\(\bar\eta\in T_pM\) be a unit vector. It was proved in
\cite[Theorem~3]{bento2023fenchel} that
\begin{equation}
    \label{eq:busemann-biconjugate-gap}
    \bigl(B^p_{\bar\eta}\bigr)^{**}
    =
    B^p_{\bar\eta}
    -\frac{1}{k}\log\!\left(\cosh\!\left(\frac{k}{2}\right)\right).
\end{equation}
Thus, even though \(B^p_{\bar\eta}\) is proper, lower semicontinuous, and
convex, it does not coincide with its biconjugate. The positive constant
\begin{equation}
    \label{eq:Ck-definition}
    C_k
    :=
    \frac{1}{k}\log\!\left(\cosh\!\left(\frac{k}{2}\right)\right)
\end{equation}
is precisely the difference between the Busemann function and its
biconjugate in this example and explicitly records the influence of the
sectional curvature.

We next introduce the geometric quantity that controls this gap. Given
\(p\in M\), define
\begin{equation}
    \label{eq:Np-definition}
    \mathcal N^p(y)
    :=
    \inf_{z\in M}
    \bigl\{H_p(y,z)-H_p(z,y)\bigr\},
    \qquad y\in M.
\end{equation}
Equivalently,
\begin{equation}
    \label{eq:Np-supremum}
    -\mathcal N^p(y)
    =
    \sup_{z\in M}
    \bigl\{H_p(z,y)-H_p(y,z)\bigr\}.
\end{equation}
Since \(H_p(p,y)=H_p(y,p)=0\), one has \(\mathcal N^p(y)\leq0\).
Therefore, \(-\mathcal N^p(y)\) is a nonnegative measure of the maximal
failure of symmetry of the coupling \(H_p\) at the point \(y\), relative to
the base point \(p\). In particular, if \(M=\mathbb R^n\), then
\[
    H_p(y,z)=\langle y-p,z-p\rangle=H_p(z,y),
\]
and hence \(\mathcal N^p\equiv0\). This motivates viewing
\(-\mathcal N^p\) as a measure of the nonlinearity detected by the Busemann
coupling.

The asymmetry in \eqref{eq:Np-supremum} also shows that, contrary to the
Euclidean case, neither inequality between \(f\) and \(f_p^{**}\) holds for
an arbitrary proper, lower semicontinuous, convex function. Indeed, fix
\(x\in M\) and consider the indicator function
\[
    \delta_{\{x\}}(z)
    :=
    \begin{cases}
        0, & z=x,\\
        +\infty, & z\neq x.
    \end{cases}
\]
This function is proper, lower semicontinuous, and geodesically convex, and
its conjugate is
\[
    \bigl(\delta_{\{x\}}\bigr)_p^{*}(y)=H_p(x,y).
\]
Consequently,
\begin{equation}
    \label{eq:indicator-biconjugate-Np}
    \bigl(\delta_{\{x\}}\bigr)_p^{**}(x)
    =
    \sup_{y\in M}\bigl\{H_p(y,x)-H_p(x,y)\bigr\}
    =
    -\mathcal N^p(x).
\end{equation}
Whenever \(\mathcal N^p(x)<0\), one therefore has
\(
    \bigl(\delta_{\{x\}}\bigr)_p^{**}(x)>\delta_{\{x\}}(x)
\).
Together with \eqref{eq:busemann-biconjugate-gap}, where the opposite strict
inequality holds, this example demonstrates that the deviation from exact
biconjugation is intrinsically two-sided. 

Accordingly, an appropriate curvature-dependent estimate must control both possible signs of the biconjugation gap rather than presuppose the Euclidean ordering. The following result determines the intrinsic nonlinearity measure exactly at an arbitrary distance from the base point in constant negative curvature. In particular, its unit-radius specialization recovers the constant appearing in \eqref{eq:busemann-biconjugate-gap}.

\begin{Proposition}
\label{prop:Np-arbitrary-radius}
Let \(p\in M\), and assume that \(M\) has constant sectional curvature
\(-k^2\), where \(k>0\). Then, for every \(y\in M\), setting
\(\rho:=d(y,p)\), one has
\begin{equation*}
\label{eq:Np-arbitrary-radius}
    \mathcal N^p(y)
    =
    -\frac{\rho}{k}
    \log\!\left(
        \cosh\!\left(\frac{k\rho}{2}\right)
    \right).
\end{equation*}
Equivalently, for \(\rho>0\),
\begin{equation*}
\label{eq:Np-arbitrary-radius-phi}
    \mathcal N^p(y)
    =
    -\frac{\rho^2}{2}
    \varphi\!\left(\frac{k\rho}{2}\right),\qquad \mbox{where}\quad 
\varphi(t)=(\log\cosh t)/t.
\end{equation*}
In particular, if \(\rho=1\), then
\[
    \mathcal N^p(y)
    =
    -\frac{1}{k}
    \log\!\left(\cosh\!\left(\frac{k}{2}\right)\right)
    =-C_k.
\]
\end{Proposition}
\begin{proof}
Set \(\rho:=d(y,p)\). If \(\rho=0\), then \(y=p\), and hence
\[
    \mathcal N^p(p)
    =
    \inf_{z\in M}
    \bigl\{H_p(p,z)-H_p(z,p)\bigr\}
    =0.
\]
Assume that \(\rho>0\), and consider the homothetically rescaled manifold \(M_{\rho}\) through the metric introduced in \eqref{eq:homothetic-metric}. From
\eqref{eq:homothetic-geometric-scaling},
\[
    d_{\rho}(y,p)=1,
    \qquad
    \operatorname{sec}M_{\rho}
    =
    \rho^2\operatorname{sec}M
    =
    -(k\rho)^2.
\]
Let \(H_p^{\rho}\) and \(\mathcal N_{\rho}^p\) denote, respectively,
the Busemann coupling and the intrinsic nonlinearity measure associated
with \(M_{\rho}\). From
\eqref{eq:Busemann-coupling-homothetic-scaling}, we have
\[
    H_p^{\rho}(z,w)
    =
    \frac{1}{\rho^2}H_p(z,w),
    \qquad z,w\in M.
\]
Consequently,
\begin{equation}
\label{eq:Np-homothetic-scaling}
\begin{aligned}
    \mathcal N_{\rho}^p(y)
    =
    \inf_{z\in M}
    \bigl\{
        H_p^{\rho}(y,z)-H_p^{\rho}(z,y)
    \bigr\} 
    =
    \frac{1}{\rho^2}\mathcal N^p(y).
\end{aligned}
\end{equation}
Set
\(
    \eta_{\rho}:=-\exp_p^{-1}y.
\)
Since \(d_{\rho}(y,p)=1\), the vector \(\eta_{\rho}\) is unit with
respect to the metric of \(M_{\rho}\), and
\(
    y=\exp_p^{\rho}(-\eta_{\rho}).
\)
Moreover, by the definition of the coupling,
\[
    H_p^{\rho}(y,z)
    =
    B_{\eta_{\rho}}^{p,\rho}(z),
    \qquad z\in M.
\]
Therefore,
\begin{align*}
    -\mathcal N_{\rho}^p(y)
    =
    \sup_{z\in M}
    \bigl\{
        H_p^{\rho}(z,y)-H_p^{\rho}(y,z)
    \bigr\} 
    =
    \sup_{z\in M}
    \bigl\{
        H_p^{\rho}(z,y)
        -
        B_{\eta_{\rho}}^{p,\rho}(z)
    \bigr\}.
\end{align*}
The last supremum is precisely the conjugate, computed on
\(M_{\rho}\), of \(B_{\eta_{\rho}}^{p,\rho}\) evaluated at \(y\).
Since \(M_{\rho}\) has constant sectional curvature
\(-(k\rho)^2\), it follows from
\cite[Theorem~3]{bento2023fenchel} that
\[
    -\mathcal N_{\rho}^p(y)
    =
    \frac{1}{k\rho}
    \log\!\left(
        \cosh\!\left(\frac{k\rho}{2}\right)
    \right).
\]
Combining this identity with
\eqref{eq:Np-homothetic-scaling} yields
\[
    \mathcal N^p(y)
    =
    -\frac{\rho}{k}
    \log\!\left(
        \cosh\!\left(\frac{k\rho}{2}\right)
    \right).
\]
Equivalently, using the definition of \(\varphi(\cdot)\), we get
\(
    \mathcal N^p(y)
    =
    -\frac{\rho^2}{2}
    \varphi\!\left(\frac{k\rho}{2}\right),
\)
and the conclusion of the proof follows by considering $\rho=1.$
\end{proof}
\begin{remark}
Note that \(C_k\) is sharp for both signs of the biconjugation deviation. If
\(d(x,p)=1\), then Proposition~\ref{prop:Np-arbitrary-radius} and
\eqref{eq:indicator-biconjugate-Np} yield
\[
    \bigl(\delta_{\{x\}}\bigr)_p^{**}(x)
    -\delta_{\{x\}}(x)
    =C_k,
\]
whereas \eqref{eq:busemann-biconjugate-gap} gives
\[
    B^p_{\bar\eta}
    -\bigl(B^p_{\bar\eta}\bigr)^{**}
    =C_k.
\]
Thus, the same curvature-dependent constant holds as an extremal deviation
on either side of exact biconjugation.
\end{remark}
The preceding proposition determines the exact global magnitude of the asymmetry of \(H_p\) under interchange of its arguments. The following corollary relates this extremal value to its radial variation, described by \(\varphi\) along the orthogonal configurations considered in Proposition~\ref{prop:Hp-asymmetry-phi}.

\begin{Corollary}[Orthogonal detection of the coupling asymmetry]
\label{cor:orthogonal-detection-Np}
Let \(M\) be a Hadamard manifold with constant sectional curvature
\(-k^2\), where \(k>0\), and \(p\in M\). Given orthogonal unit
vectors \(\eta,\theta\in T_pM\) and \(r,s>0\), define
$$
    y:=\exp_p(-s\eta),
    \qquad  z:=\exp_p(r\theta).
$$
If \(0<r<s\), then
\begin{equation}
\label{eq:orthogonal-asymmetry-Np}
\begin{aligned}
0
<
H_p(z,y)-H_p(y,z)\
=
rs\bigl[\varphi(ks)-\varphi(kr)\bigr]\
\leq
-\mathcal N^p(y)
=
\frac{s^2}{2}
\varphi\left(\frac{ks}{2}\right).
\end{aligned}
\end{equation}
In particular, if \(s=1\), then
$
0<r\bigl[\varphi(k)-\varphi(kr)\bigr]
    \leq
    C_k
    =
(1/2)\varphi\!\left(k/2\right).
$

\end{Corollary}

\begin{proof}
Proposition~\ref{prop:Hp-asymmetry-phi} gives
\begin{equation}\label{eq:1111}
    H_p(z,y)-H_p(y,z)
    =
rs\bigl[\varphi(ks)-\varphi(kr)\bigr].
\end{equation}
Because \(\varphi(\cdot)\) is strictly
increasing and \(0<r<s\), it follows that the quantity in \eqref{eq:1111} is strictly positive. Moreover, by the definition of
\(\mathcal N^p\), one has
$$
    H_p(z,y)-H_p(y,z)
    \leq
    \sup_{w\in M}
    \bigl\{H_p(w,y)-H_p(y,w)\bigr\}
    =
    -\mathcal N^p(y).
$$
Since \(d(y,p)=s\), Proposition~\ref{prop:Np-arbitrary-radius} yields

$$
    -\mathcal N^p(y)
    =
    \frac{s}{k}
    \log\!\left(\cosh\!\left(\frac{ks}{2}\right)\right)
    =
    \frac{s^2}{2}
    \varphi\!\left(\frac{ks}{2}\right).
$$
The unit-radius statement follows by setting \(s=1\), which concludes the proof.
\end{proof}


\begin{remark}[A common geometric source]
The preceding discussion also makes explicit the connection with the
comparison of Busemann-based subdifferentials developed in the previous
section. Up to differences in signs, normalizations, and parametrizations,
the other constructions considered there are based on a fixed Busemann
support. More precisely, if
\[
    y=\exp_p(-s\eta),
    \qquad s>0,\quad \|\eta\|=1,
\]
then the corresponding fixed-direction support can be written as
\[
    sB_\eta^p(z)=H_p(y,z).
\]
By contrast, the support inequality defining our subdifferential involves
\[
    H_p(z,y)-H_p(x,y),
\]
and, in particular, at \(x=p\), its support term is \(H_p(z,y)\).
Consequently, the structural difference between the two types of
subdifferentials is governed by the lack of reciprocity of the Busemann
coupling,
\[
    H_p(z,y)-H_p(y,z).
\]

In constant sectional curvature \(-k^2\), the orthogonal configuration
considered above gives
\[
    H_p(z,y)-H_p(y,z)
    =
    rs\bigl[\varphi(ks)-\varphi(kr)\bigr],
    \qquad
    \varphi(t)=\frac{\log(\cosh t)}{t}.
\]
Thus, the strict monotonicity of \(\varphi\), which yields the separation
examples between the subdifferentials, measures a particular radial
manifestation of the same geometric asymmetry that is globalized by
\[
    -\mathcal N^p(y)
    =
    \sup_{z\in M}
    \bigl\{H_p(z,y)-H_p(y,z)\bigr\}.
\]
Hence, the non-reciprocity of \(H_p\) is the common geometric mechanism
that both distinguishes our subdifferential from the fixed-direction
Busemann constructions and controls the deviation from exact
biconjugation. In the Euclidean case this defect vanishes, since \(H_p\)
is symmetric, and both distinctions disappear.
\end{remark}

We next show that \(\mathcal N^p\) controls both possible signs of the
difference between a function and its biconjugate. For \(x\in M\), let
\begin{equation*}
    \label{eq:dual-points-subdifferential}
    \mathcal Y_p^f(x)
    :=
    \bigl\{Y_p(x,v):v\in\partial B^p f(x)\bigr\},
\end{equation*}
where \(Y_p(x,v)\) is the dual point associated with \(v\) by the
primal--dual correspondence established in the preceding section.

\begin{Theorem}
    \label{thm:biconjugate-gap-Np}
    Let \(f:M\to\mathbb R\cup\{+\infty\}\) be proper and let
    \(x\in\operatorname{dom}f\). Then
    \begin{equation}
        \label{eq:biconjugate-gap-lower}
        \mathcal N^p(x)
        \leq
        f(x)-f_p^{**}(x).
    \end{equation}
    If, in addition, \(\partial B^p f(x)\neq\varnothing\), then, for every
    \(y\in\mathcal Y_p^f(x)\),
    \begin{equation}
        \label{eq:biconjugate-gap-chain}
        \mathcal N^p(x)
        \leq
        f(x)-f_p^{**}(x)
        \leq
        H_p(x,y)-H_p(y,x)
        \leq
        -\mathcal N^p(y).
    \end{equation}
    Consequently,
    \begin{equation}
        \label{eq:biconjugate-gap-infimum}
        f(x)-f_p^{**}(x)
        \leq
        \inf_{y\in\mathcal Y_p^f(x)}
        \bigl\{-\mathcal N^p(y)\bigr\}.
    \end{equation}
\end{Theorem}

\begin{proof}
    By the definition of the conjugate, for every \(w\in M\),
    \(
        f_p^*(w)
        \geq
        H_p(x,w)-f(x).
    \)
    Therefore,
    \[
        H_p(w,x)-f_p^*(w)
        \leq
        f(x)+H_p(w,x)-H_p(x,w).
    \]
    Taking the supremum over \(w\in M\) and using
    \eqref{eq:Np-supremum}, we obtain
    \(
        f_p^{**}(x)
        \leq
        f(x)-\mathcal N^p(x),
    \)
    which is equivalent to \eqref{eq:biconjugate-gap-lower}.

    Now fix \(y\in\mathcal Y_p^f(x)\). By the primal--dual characterization of
    the Busemann subdifferential established in the preceding section, the
    Fenchel-Young equality holds:
    \[
        f(x)+f_p^*(y)=H_p(x,y).
    \]
    On the other hand, the definition of the biconjugate gives
    \[
        f_p^{**}(x)
        =
        \sup_{w\in M}\bigl\{H_p(w,x)-f_p^*(w)\bigr\}
        \geq
        H_p(y,x)-f_p^*(y).
    \]
    Subtracting the latter inequality from the Fenchel-Young equality yields
    \[
        f(x)-f_p^{**}(x)
        \leq
        H_p(x,y)-H_p(y,x).
    \]
    Furthermore, taking \(z=x\) in \eqref{eq:Np-definition}, we obtain
    \(
        \mathcal N^p(y)
        \leq
        H_p(y,x)-H_p(x,y),
    \)
    or equivalently,
    \[
        H_p(x,y)-H_p(y,x)
        \leq
        -\mathcal N^p(y).
    \]
    Combining these inequalities with \eqref{eq:biconjugate-gap-lower} proves
    \eqref{eq:biconjugate-gap-chain}. Since the estimate holds for every
    \(y\in\mathcal Y_p^f(x)\), taking the infimum yields
    \eqref{eq:biconjugate-gap-infimum}.
\end{proof}

\begin{Remark}
    \label{rem:infimum-domain}
    The infimum in \eqref{eq:biconjugate-gap-infimum} is taken over the dual
    points generated by \(\partial B^p f(x)\), not over all of \(M\). This
    restriction is essential: the upper estimate for
    \(f(x)-f_p^{**}(x)\) relies on equality in the Fenchel-Young inequality,
    which is guaranteed precisely by the primal--dual points associated with
    Busemann subgradients. By contrast, the lower estimate
    \eqref{eq:biconjugate-gap-lower} holds for every proper function and
    involves \(\mathcal N^p(x)\), rather than a dual point generated by the
    subdifferential.

The preceding results show that the deviation from exact biconjugation is not
an unrelated error term: both of its possible signs are controlled by the
failure of reciprocity of the Busemann coupling. In constant negative
curvature, this defect is measured sharply by \(C_k\); in the orthogonal
hyperbolic configuration, its radial variation is detected by the strictly
increasing function \(\varphi(\cdot)\). Thus, the same measure of nonlinearity that
separates our subdifferential from the fixed-direction Busemann constructions
also controls the failure of exact biconjugation.
\end{Remark}
\begin{remark}\label{rem:Goodwin-adjoint-biconjugation}
Goodwin et al.\ also establish a Fenchel--Moreau-type result; for
details and notation, see
\cite[Proposition~3.5]{goodwin2026stochastic}. Their result, however,
concerns an adjoint biconjugation between two distinct function spaces,
rather than the iteration of a single conjugation operator. Indeed,
\(f^\bullet\) maps functions on the Hadamard space \(X\) to functions
on the boundary cone \(CX^\infty\), whereas the adjoint transform
\(g^\circ\) maps functions on \(CX^\infty\) back to \(X\). Thus, their
biconjugate is the composition
\[
   \overline{\mathbb R}^{\,X}
   \xrightarrow{\ \bullet\ }
   \overline{\mathbb R}^{\,CX^\infty}
   \xrightarrow{\ \circ\ }
   \overline{\mathbb R}^{\,X}.
\]
Accordingly, the term \emph{biconjugate} in their framework refers to
\((f^\bullet)^\circ\), not to the result of applying
\((\cdot)^\bullet\) twice.
When \(X=M\) is a Hadamard manifold and \(p\in M\) is fixed, the
identification \(CM^\infty\simeq T_pM\) yields
\[
   \overline{\mathbb R}^{\,M}
   \xrightarrow{\ \bullet_p\ }
   \overline{\mathbb R}^{\,T_pM}
   \xrightarrow{\ \circ_p\ }
   \overline{\mathbb R}^{\,M},
\]
with coupling
\[
   c_p(x,v)
   =
   \|v\|B^p_{-v/\|v\|}(x)
   =
   H_p\bigl(\exp_p(v),x\bigr),
   \qquad v\neq0,
\]
and \(c_p(x,0)=0\). Hence, \((f^\bullet)^\circ\) is generated by an
adjoint pair of transforms that preserves the primal--dual ordering of
this coupling. Their Fenchel--Moreau-type result therefore concerns the
closure induced by this adjoint pair.

By contrast, our conjugation maps functions on \(M\) to functions on
\(M\), and the biconjugate \(f_p^{**}\) is obtained by reapplying the
same operator:
\[
    f_p^{**}(x)
    =
    \sup_{y\in M}
    \bigl\{H_p(y,x)-f_p^*(y)\bigr\}.
\]
While the first conjugation is defined through \(H_p(x,y)\), its second
application introduces \(H_p(y,x)\) and therefore interchanges the
arguments of \(H_p\). Consequently,  the two constructions
address different biconjugation problems. This distinction disappears
when \(H_p\) is symmetric, as in the Euclidean setting.
\end{remark}

\section{Conclusion}\label{Sec6}
This paper continues and expands the research line initiated in
\cite{bento2022combinatorial} and advanced in
\cite{bento2023fenchel}. We introduced an intrinsic subdifferential
framework on Hadamard manifolds based on Busemann functions and the
base-point-dependent coupling \(H_p\). Its central geometric ingredient
is the pair of mutually inverse mappings
$$ Y_p(x,\cdot):T_xM\longrightarrow M
 \qquad\text{and}\qquad V_p(x,\cdot):M\longrightarrow T_xM,
$$
which identify each dual point with a tangent vector whose norm records
its distance from the base point and whose direction is determined by
the corresponding asymptotic geodesic class. Through this intrinsic
parametrization, attainment in the definition of the conjugate becomes
a tangent-space subdifferential characterization. Consequently, for
arbitrary proper functions, the \(B^p\)-subdifferential characterizes
precisely the primal-dual pairs at which equality holds in the
associated Fenchel-Young inequality, while its zero vector yields the
exact Fermat rule. We also established covariance under isometries and
obtained an explicit description for a class of radial convex functions,
complementing the radial conjugacy theory developed in
\cite{bento2023fenchel}.

Beyond these structural properties, we characterized the natural
finite-valued class on which the new subdifferential is everywhere
nonempty. More precisely,
$$
    \partial B^p f(x)\neq\varnothing
    \quad\text{for every }x\in M
$$
if and only if \(f\) admits a pointwise-attained \(H_p\)-envelope
representation. Hence, nonemptiness is equivalent to the existence of
global supports generated by the correctly oriented sections
\(H_p(\cdot,y)\), active at every point. This class includes finite
maxima of such sections, the radial convex functions considered here,
and the nonsmooth functions \(\lambda d(p,\cdot)\), \(\lambda>0\).
In the complementary differentiable setting, we established a rigidity
property at the distinguished base point: on hyperbolic space,
$$
\begin{aligned}
    \partial B^p f(p)\neq\varnothing
    &\quad\Longleftrightarrow\quad
    \partial B^p f(p)=\{0_p\}\\
    &\quad\Longleftrightarrow\quad
    p\in\operatorname*{argmin}_{z\in M}f(z)\\
    &\quad\Longleftrightarrow\quad
    f(p)+f_p^*(p)=H_p(p,p)=0.
\end{aligned}
$$
Thus, the new theory provides both a functional characterization of
global subdifferentiability and a sharp primal-dual certificate of
global minimality at the base point.

Two explicit geodesically convex functions on the Poincaré disk
establish strict noninclusions between our construction and
fixed-direction Busemann subdifferentials. These examples show that the
latter cannot, in general, replace the \(B^p\)-subdifferential in the
Fenchel--Young equality characterization, although all these
constructions coincide in Euclidean spaces. Importantly, the function
in the first separation example is generated by the reversed section
\(H_p(y,\cdot)\), whereas the natural generators of the present theory
are the sections \(H_p(\cdot,y)\). The example therefore reveals a
separation between two support theories induced by the orientation of
the coupling, rather than a failure of the \(B^p\)-subdifferential on
its own generating class.

Our analysis further identifies the asymmetry of \(H_p\) under
interchange of its arguments as the common geometric mechanism
underlying both this incomparability and the failure of exact Fenchel
biconjugation. In constant negative curvature, a strictly increasing
scalar profile determines the sign and radial variation of this
asymmetry. Its maximal extent, encoded by the intrinsic nonlinearity
measure \(-\mathcal N^p\), now acquires a precise
mathematical role. We established a  lower bound for the
biconjugation gap and, at points admitting \(B^p\)-subgradients,
corresponding two-sided estimates.  This problem concerns the iteration of the same conjugation
operator, which reverses the arguments of \(H_p\), and is therefore
distinct from adjoint biconjugation frameworks that preserve the
primal--dual orientation of their coupling.
Furthermore, if \(M\) has constant sectional curvature \(-k^2\) and
\(\rho=d(y,p)>0\), we obtained the exact formula

$$
    \mathcal N^p(y)
    =
    -\frac{\rho}{k}
    \log\!\left(
        \cosh\!\left(\frac{k\rho}{2}\right)
    \right).
$$
At unit distance from the base point, this identity recovers the
curvature-dependent constant appearing in the biconjugation formula for
Busemann functions, and extremal deviations of both signs attain this
constant. Consequently, a direct extension of the classical
Fenchel-Moreau theorem cannot be expected without additional geometric
or functional assumptions.
These results provide a basis for further developments of
Busemann-based convex analysis. Natural directions include extending
the exact constant-curvature formula for \(-\mathcal N^p\) to Hadamard
manifolds with variable pinched curvature,
$$
    -b^2
    \leq
    \operatorname{sec}_M
    \leq
    -a^2
    <0,
    \qquad
    0<a<b,
$$
developing broader verifiable conditions ensuring attainment and
nonemptiness for \(H_p\)-envelopes, and investigating the role of the
proposed subdifferential in duality theory and optimization algorithms
on Hadamard manifolds.


\section*{Acknowledgments}
This work was funded by the CNPq Grants  302156/2022-4.

\section*{Declarations}
The authors declare that they have no conflict of interest.


\end{document}